\documentclass[leqno, letterpaper]{amsart}
\usepackage{ marvosym }
\usepackage{amssymb, mathtools, amsmath}
\usepackage{tikz, tikz-cd, graphicx}
\usepackage[normalem]{ulem}
\usepackage{color}
\usepackage{hyperref}
\usetikzlibrary{matrix,arrows,decorations.pathmorphing}

\usepackage{hyperref}
\usepackage{scalerel,stackengine}
\stackMath
\newcommand\reallywidehat[1]{%
\savestack{\tmpbox}{\stretchto{%
  \scaleto{%
    \scalerel*[\widthof{\ensuremath{#1}}]{\kern-.6pt\bigwedge\kern-.6pt}%
    {\rule[-\textheight/2]{1ex}{\textheight}}
  }{\textheight}%
}{0.5ex}}%
\stackon[1pt]{#1}{\tmpbox}%
}
\newcommand{\Addresses}{{
  \bigskip
  \footnotesize

  Damjan PI\v{S}TALO, \textsc{Department of Mathematics, University of Luxembourg,
    2, place de l’Université
L-4365 Esch-sur-Alzette, Luxembourg}, E-mail: \textit{damjan.pistalo@uni.lu}.
}}

\usepackage{thmtools}
\usepackage{thm-restate}

\usepackage{amscd}
\usepackage{amsmath}
\usepackage{amsthm}
\usepackage{mathrsfs}
\usepackage{latexsym}
\usepackage{amssymb}
\usepackage{amsfonts}
\theoremstyle{plain}

\usepackage{tikz}
\usetikzlibrary{arrows,chains,matrix,positioning,scopes,snakes}

\tikzset{>=stealth',every on chain/.append style={join},
         every join/.style={->}}
\tikzset{
    >=stealth',
    punkt/.style={
           rectangle,
           rounded corners,
           draw=black, very thick,
           text width=6.5em,
           minimum height=2em,
           text centered},
    pil/.style={
           ->,
           thick,
           shorten <=2pt,
           shorten >=2pt,}
}

\usepackage[T1]{fontenc}
\newcommand{\bee}{\begin{enumerate}}
\newcommand{\eee}{\end{enumerate}}
\newcommand{\benn}{\begin{equation*}}
\newcommand{\eenn}{\end{equation*}}
\newcommand{\be}{\begin{equation}}
\newcommand{\ee}{\end{equation}}
\newcommand{\bean}{\begin{eqnarray}}
\newcommand{\eean}{\end{eqnarray}}
\newcommand{\bea}{\begin{eqnarray*}}
\newcommand{\eea}{\end{eqnarray*}}

\newcommand{\op}[1]{\!\!\mathop{\rm ~#1}\nolimits}

\newcommand{\cyclic}{\mathop{\kern0.9ex{{+}
\kern-2.15ex\raise-.25ex\hbox{\Large\hbox{$\circlearrowright$}}}}\limits}

\newtheorem{theo}{Theorem}
\newtheorem{rem}[theo]{Remark}
\newtheorem{prop}[theo]{Proposition}

\newtheorem{cor}[theo]{Corollary}
\newtheorem{ex}[theo]{Example}
\newtheorem{defi}[theo]{Definition}

\newtheorem{theorem}[theo]{Theorem}
\newtheorem{lemma}[theo]{Lemma}

   \def\i{\imath}

\usepackage{stackrel}
\usepackage{epigraph}

\usepackage{lmodern}
\usepackage{here}
\usepackage{url}
\usepackage{hyperref}

\usepackage{hyperref}
\usepackage{tikz-cd}

\newsavebox{\pullback}
\sbox\pullback{%
\begin{tikzpicture}%
\draw (0,0) -- (1ex,0ex);%
\draw (1ex,0ex) -- (1ex,1ex);%
\end{tikzpicture}}

   \numberwithin{equation}{section}

\tikzcdset{scale cd/.style={every label/.append style={scale=#1},
    cells={nodes={scale=#1}}}}

\title{Pro-nilpotently extended dgca-s and SH Lie-Rinehart pairs}
\author{Damjan Pi\v{s}talo}
\date{\today}
\begin{document}

\maketitle

\begin{abstract}
Category of pro-nilpotently extended differential graded commutative algebras is introduced. Chevalley-Eilenberg construction provides an equivalence between its certain full subcategory and the opposite to the full subcategory of strong homotopy Lie Rinehart pairs with strong homotopy morphisms, consisting of pairs $(A,M)$ where $M$ is flat as a graded $A$-module. It is shown that pairs $(A,M)$, where $A$ is a semi-free dgca and $M$ a cell complex in $\op{Mod}(A)$, form a category of fibrant objects by proving that their  Chevalley-Eilenberg complexes form  a category of cofibrant objects.
\end{abstract}

\tableofcontents

\section{Introduction}
Lie algebras have a solid homotopy theory. Its development dates back Quillen's work on rational homotopy theory \cite{Quil}, which, among many other things, establishes an adjunction  
between dg Lie algebras (denoted by $\tt DGL$) and a certain class of differential graded cocommutative coalgebras (denoted by $\tt DGC$), whose right adjoint sends a dg Lie algebra to its homological Chevalley-Eilenberg complex. In \cite{Hin}, Hinich puts cofibrantly generated model structures on both categories, promoting the above adjunction to a Quillen equivalence. Fibrant objects in $\tt DGC$ are exactly Chevalley-Eilenberg complexes of Lie $\infty$-algebras, and morphisms between them are equivalently $\infty$-morphisms of  Lie $\infty$-algebras. In \cite{Pri}, J. Pridham dualized differential graded cocommutative coalgebras into so-called $\mathbb{Z}$-graded pro-Artinian chain algebras (denoted by  $dg_\mathbb{Z}\hat{\mathcal{C}}_k$) which form a fibrantly generated model category, whose cofibrant objects are cohomological Chavelley-Eilenberg complexes of Lie $\infty$-algebras. Further on, Pridham shows that weak equivalences between cofibrant objects are precisely those maps whose weight-zero part in the grading of the completed symmetric tensor product dualizes to a weak equivalence of chain complexes; equivalently, that weak equivalences between fibrant objects in $\tt DGC$ are $\infty$-quasi-isomorphisms of Lie $\infty$-algebras in the sense of \cite{Valette}. 

The above story has been partially generalized to Lie Rinehart pairs mainly due to the work of G. Vezzosi (\cite{Vez}) and J. Nuiten (\cite{Nui1},\cite{Nui2}). More concretely, it is shown that the category of dg Lie Rinehart pairs $(A,M)$ with a fixed base algebra $A$ (denoted by $\op{LieAlgd_A^{dg}}$) is a semi-model category, and that its homotopy theory is equivalent to that of formal moduli problems over $A$. The equivalence builds on the relation between dg Lie algebras and deformation problems studied in the above cited work of Hinich and Pridham and elsewhere.

Homotopy theory for Lie-Rinehart pairs over a variable basis is implied in the construction of the BV-BRST complex, which can be interpreted (see for example \cite{HT}) as a two-step resolution of a Lie-Rinehart pair $(A,M)$ by a strong homotopy Lie Rinehart pair $(QA,P)$  in the sense of \cite{Vit}, where $QA\to A$ is a Koszul-Tate resolution of $A$, and $P\to M$ a free resolution of $M$ as a $QA$-module. Although different variants of such resolutions of Lie Rinehart pars have been formalized for example in \cite{Kje},  \cite{KLGS}, and \cite{LLG}, to the best of my knowledge, they have not yet been embedded into a solid homotopy theory. 

In the present paper, I define differential pro-graded modules and algebras as pro-object in the categories of graded modules and algebras, together with a differential which satisfies the Leibniz rule in an appropriate sense. Category of pro-nilpotently extended or fat dgca-s (denoted by ${\tt fcdga}(k)$) is defined as the full subcategory of the arrow category of differential pro-graded commutative algebras consisting of maps $A\to A_0$, with $A_0$ a non-positively graded dgca, and $A$ an inverse system of unbounded graded-commutative algebras which are nilpotent extensions of $A_0$. Denote by ${\tt SHRL}_{\op{flat}}(k)$ the full subcategory of strong homotopy Lie-Rinehart pairs $(A,M)$ with $M$ flat as a graded $A$-module.  Chevalley-Eilenberg construction provides an equivalence between  ${\tt SHLR}_{\op{flat}}(k)$, and the opposite to a certain full subcategory of ${\tt fcdga}(k)$.

Generalizing an above-mentioned Pridham's result, I show that  Chevalley-Eilenberg complexes of strong homotopy Lie-Rinehart pairs $(A,M)$, with $A$ a semi-free dgca and $M$ a cell complex in $\op{Mod}(A)$, form a category of cofibrant objects, in which the map dual to a SH morphism $(A,M)\to(B,N)$ is a weak equivalence if its weight-zero components $B\to A$ and $M\to A\otimes_BN$ are both quasi-isomorphisms.
\subsection*{Notation}

Throughout the paper, $k$ is a fixed field of characteristic zero.

To avoid confusion, graded objects  (modules, algebras)  are denoted by typewriter capital letters ($\tt A, B, M, N,\ldots$), whereas differential graded objects are denoted by the italic capital letters ($A, B, M, N,\ldots$).  All the differential graded objects are cohomological. 

Given a non-positively graded graded-commutative $k$-algebra ${\tt A}=\oplus_{n\in\mathbb{Z}_{\leq 0}}A_{n}$, ${\tt Mod}({\tt A})$ denotes the category of $\mathbb{Z}$-graded $k$-vector spaces with an $\tt A$-action in the sense of \cite[VII.4]{MacL}.  

Given $M\in{\tt Mod}({\tt A})$, degree of a homogeneous element $m\in \tt M$ is denoted by $|m|$. The free graded symmetric $\tt A$-algebra of $\tt M$ is 
$$\tt \op{Sym}_{\tt A}M=\oplus_{k\in\mathbb{N}}{\tt M}^{\otimes_{\tt A} k}/{\mathcal{I}},$$ where $\mathcal{I}$ is the ideal generated by $\{m\otimes n- (-1)^{|m|\cdot|n|} n\otimes m:m,n\in\tt M\}$.
Its multiplication (the graded symmetric tensor product) is denoted by $-\odot -$.  The index $k$ is referred to as weight. $\tt A$-module of weight $k$ elements in $\tt \op{Sym}_{\tt A}M$ is denoted by $\op{Sym}_{\tt A}^k\tt M$.

The category of non-positively graded differential  graded-commutative algebras over $k$ is denoted by ${\tt dgca}(k).$ Given $A\in{\tt dgca}(k)$, ${\tt Mod}(A)$ denotes the category of unbounded cochain complexes of $k$-vector spaces with an $A$-action. Explicitly, a cochain complex $({\tt M},d)$ is an $A$-module if $\tt M$ is a graded $\tt A$-module, and the differential satisfies the graded Leibniz identity
\begin{equation}\label{leib1}d(a\cdot m)=d_Aa\cdot m+(-1)^{|a|}a\cdot dm.\end{equation}
Morphisms are required to preserve the $\tt A$-module structure and intertwine the differential.

The group of all the permutations of a set with $k$ elements is denoted by $\Sigma_k$. The set of $(l,m)$-unshuffles is denoted by $\op{Sh}(l,m)\subset \Sigma_{l+m}$. 

Given a $k$-permutation $\sigma\in \Sigma_k$ together with $m_1,\ldots, m_k\in\tt M$, the number $|m_1,\ldots,m_k|_{\sigma}\in \{\pm 1\}$ is implicitly defined by
$$m_{\sigma(1)}\odot\ldots\odot m_{\sigma(k)}=(-1)^{|m_1,\ldots,m_k|_{\sigma}}m_{1}\odot\ldots\odot m_{1}.$$

$\mathbb{N}$ denotes the set of natural numbers including zero. The set of natural numbers without zero is denoted by $\mathbb{N}_{>0}.$

\subsection*{Acknowledgments} I am grateful to J. Nuiten for warning me against pulling back Lie algebroids.


\section{Pro-finite and flat graded modules -- duality}


${\tt M}\in {\tt Mod}({\tt A})$ is called free if it is isomorphic to the tensor product of $\tt A$ with a graded vector space (${\tt M}={\tt A} \otimes_k \mathbb{V})$. It is free finite if the graded vector space $\mathbb{V}$ is finite dimensional in each degree and non-zero in at most finitely many degrees. The full subcategory of free finite $\tt A$-modules (throughout the text referred to as \emph{finite $\tt A$-modules} for simplicity) is denoted by $\tt Mod_{\op{finite}} (A)$.

 ${\tt Mod}({\tt A})$ is a closed symmetric monoidal Abelian category, the compatibility condition being that the tensor product $-\otimes_{\tt A}-$ is additive in both arguments. To prove that the later category is indeed Abelian, notice first that the category ${\tt gVec}(k)$ of $\mathbb{Z}$-graded vector spaces is Abelian (\cite[12.16.2]{Stacks}). By \cite{Ard}, the category ${\tt Mod}({\tt A})$ is Abelian as well, as the monad ${\tt A}\otimes_k -:{\tt gVec}(k)\to {\tt gVec}(k)$ of the free-forgetful adjunction
$${\tt A}\otimes_k -:{\tt gVec}(k)\rightleftarrows  {\tt Mod}({\tt A}): \op{For}$$
is additive and preserves cokernels.


With this, one can define flat $\tt A$ modules:
\begin{defi}
$\tt M \in Mod(A)$ is called flat if the tensor product functor $-\otimes_{\tt A}\tt M:{\tt Mod}({\tt A})\to {\tt Mod}({\tt A})$ is exact.
\end{defi}

Below generalizations of classical results are immediate:

\begin{prop}Every free $\tt A$-module is projective, and every projective $\tt A$-module is flat.
\end{prop}
\begin{theorem} (Lazard's theorem)  $\tt M\in {\tt Mod}({\tt A})$ is flat if and only if it is the colimit of a directed system of finite $\tt A$-modules.
\end{theorem}

It follows that flat $\tt A$-modules are dualizable in a specific sense:

\begin{lemma}\label{grduality} The full subcategory $\tt Mod_{\op{flat}} (A)\subset \tt Mod (A)$ of flat $\tt A$-modules is equivalent to the opposite of the category $\op{Pro}(\tt Mod_{\op{finite}} (A))$ of pro finite $\tt A$-modules.
\end{lemma}
\begin{proof}
By the Lazard's theorem, the functor
\begin{equation}\label{equiv}\varinjlim:\op{Ind}(\tt Mod_{\op{finite}} (A))\to Mod_{\op{flat}} (A)\end{equation}
from the category of ind finite $\tt A$-modules to the category of flat $\tt A$ modules, given by mapping a direct system to its colimit, is well-defined and essentially surjective. It is also full and faithful. Namely, for $({\tt M}_\alpha)_\alpha, ({\tt N}_\beta)_\beta\in\op{Ind}(\tt Mod_{\op{finite}} (A))$, 
\begin{equation}\label{indhom}\op{Hom}_{\tt Mod (A)}(\varinjlim_\alpha {\tt M}_\alpha, \varinjlim_\beta {\tt N}_\beta)=\varprojlim_\alpha\op{Hom}_{\tt Mod (A)}( {\tt M}_\alpha, \varinjlim_\beta {\tt N}_\beta)=\varprojlim_\alpha\varinjlim_\beta\op{Hom}_{\tt Mod (A)} ({\tt M}_\alpha,  {\tt N}_\beta),\end{equation}
since free finite $\tt A$-modules are compact objects in $\tt Mod(A).$ 

As the functor
$$\underline{\op{Hom}}_{\tt Mod (A)}(-,{\tt A}):\tt Mod_{\op{finite}} (A)^{\op{op}}\to \tt Mod_{\op{finite}} (A),$$
is an equivalence, it follows that
$$\tt Mod_{\op{flat}} (A)\cong\op{Ind}(\tt Mod_{\op{finite}} (A))\cong \op{Pro}(\tt Mod_{\op{finite}} (A)^{\op{op}})^{\op{op}}\cong \op{Pro}(\tt Mod_{\op{finite}} (A))^{\op{op}}.$$
\end{proof}

\section{Differential pro-graded modules and algebras}
Dualizing differential graded modules requires care.
Consider $A=({\tt A},d_A)\in{\tt dgca}(k)$. $A$-module is called \emph{graded-flat} if the underlying graded $\tt A$-module is flat. The full subcategory of graded-flat $\tt A$-modules is denoted by ${\tt Mod }_{\op{flat}}(A)$.
Given $M=({\tt M},d)\in {\tt Mod}_{\op{flat}}(A)$ and the isomorphism $\tt M\cong \varinjlim_{\alpha}M_{\alpha}$ of Lazard's theorem, the differential on $\tt M$ does not in general restrict to $\tt M_\alpha$-s. One should, rather,  endow the corresponding ind finite $\tt A$-module and its dual pro finite $\tt A$ module with appropriate differentials. As categories of both pro and ind $k$-vector spaces are Abelian (see for example \cite{isa}), it makes sense to speak of chain complexes there. However, it is not a priory clear what should be pro and ind versions of the Leibniz identity (\ref{leib1}). In the sequel, I only detail a notion of differential pro-graded $A$-modules, as the construction of ind-$A$-modules is analogous. To begin with,  recall a characterization of morphisms in pro-categories from the book of Marde\v{s}i\'{c} and Segal \cite{MS}:

Let $\mathcal{C}$ be a category; $A$ and $B$ directed sets;  $(X_\alpha,p_{\alpha,\alpha'})$ and $(Y_\beta,p_{\beta,\beta'})$ inverse systems in $\mathcal{C}$. A morphism of inverse systems $(X_\alpha)\to(Y_\beta)$ consists of a function $\phi:B\to A$ and of morphisms $f_\beta:X_{\phi(\beta)}\to Y_\beta,$ one for each $\beta\in B$, such that for all $\beta,\beta'\in B$, $\beta>\beta'$ there exists $\alpha>\phi(\beta),\phi(\beta')$ for which the diagram
\begin{center}
\begin{tikzcd}
X_{\phi(\beta)}\arrow[d,"f_\beta"]&\arrow[l,swap,"p_{\alpha,\phi(\beta)}"]X_{\alpha}\arrow[r,"p_{\alpha,\phi(\beta')}"]&X_{\phi(\beta')}\arrow[d,"f_{\beta'}"]\\
Y_{\beta}\arrow[rr,"p_{\beta,\beta'}"]&&Y_{\beta'}
\end{tikzcd}
\end{center}
commutes. Morphisms of inverse systems are denoted by $(f_\beta,\phi).$ We say that morphisms of inverse systems $(f_\beta,\phi)$ and $(g_{\beta},\psi)$ are equivalent if for every $\beta\in B$ there exists $\alpha\in A$, $\alpha>\phi(\beta),\psi(\beta)$ such that the diagram
\begin{center}
\begin{tikzcd}
X_{\phi(\beta)}\arrow[dr,"f_\beta"]&\arrow[l,swap,"p_{\alpha,\phi(\beta)}"]X_{\alpha}\arrow[r,"p_{\alpha,\psi(\beta)}"]&X_{\psi(\beta)}\arrow[dl,"g_{\beta}"]\\
&Y_{\beta}&
\end{tikzcd}
\end{center}
commutes. Finally, morphisms in a pro-category are characterized as equivalence classes of morphisms of inverse systems.

Let $f$ be a morphism in $\op{Pro}(C)$ represented by a map of inverse systems $(f_{\beta},\phi)$. Denote for $\alpha>\phi(\beta)$ $f_{\alpha,\beta}=f_{\phi(\beta),\beta}\circ p_{\alpha,\phi(\beta)}$. The family $(f_{\alpha,\beta})_{\beta\in B, \alpha\geq \phi(\beta)}$ will be called a \emph{representing family} of $f$. Observe that $(f_{\alpha,\beta})_{\beta\in B, \alpha\geq \phi(\beta)}$ and $(f'_{\alpha,\beta})_{\beta\in B, \alpha\geq \psi(\beta)}$ represent the same morphism if and only if for any $\beta\in B$ there exists $\alpha\in A$ such that $f_{\alpha,\beta}=f'_{\alpha,\beta}.$ In a slight abuse of terminology, I identify maps with their representing families, writing $f=(f_{\alpha,\beta})$. For $X=(X_\alpha,p_{\alpha,\alpha'})\in \op{Pro}\mathcal{C}$, $\op{id}_X=(p_{\alpha\alpha'})$.

 Coming back to graded $\tt A$-modules, differential is defined by means of a representing family. For this, the Leibniz identity \ref{leib1} is employed relative to underlying maps $p_{\alpha'\alpha}$. Generally, derivations relative to a map are defined as follows:
\begin{defi}Given a differential graded $k$-algebra $A=({\tt A},d_A)$ and a map $f:\tt M\to N$ in $\tt Mod(A)$, a morphism of graded $k$-modules $d_f:\tt M\to N[-1]$ is called a \emph{derivation over $f$} if for all $a\in A$
$$d_f(a\cdot m)=d_A(a)\cdot f(m)+(-1)^{|a|}a\cdot d_f(m).$$ The set of all derivations over $f$ is denoted by $\op{Der}_f(\tt M, N).$ 
\end{defi}
Crucially, derivations over graded $\tt A$-module maps compose with the $\tt A$-module maps. Proof is straight-forward.
\begin{prop}\label{prop2}
Let $f:\tt L\to M$ and $g:\tt M\to  N$ be morphisms of $\tt A$-modules.
\begin{enumerate}
\item Given a derivation $d_f:\tt L\to M[-1]$  over $f$, $g[-1]\circ d_f$ is a derivation over $g\circ f$.
\item Given  a derivation $d_g:\tt M\to N[-1]$ over $g$, $d_g\circ f$ is a derivation over $g\circ f$.
\end{enumerate}
\end{prop}
Finally, derivations of a pro-graded $A$-module are defined as follows:
\begin{defi}Given a differential graded $k$-algebra $A=({\tt A},d_A)\in {\tt dgca}(k)$ and a graded $\tt A$-module $\tt M$, derivations of $M$ are the following endomorphisms of pro-graded $k$-vector spaces:
$$\op{Der}({\tt M})=\varprojlim_{\beta}\varinjlim_{\alpha\geq \beta}\op{Der}_{p_{\alpha\beta}} ({\tt M}_\alpha,  {\tt M}_\beta)\subset \varprojlim_{\beta}\varinjlim_{\alpha\geq \beta}\op{Hom}_{k} ({\tt M}_\alpha,  {\tt M}_\beta[-1])=\op{Hom}_{\op{Pro}({\tt gMod}(k))}(\tt M, M[-1]).$$
More generally, given a map of pro-graded $\tt A$-modules $f=(f_{\alpha,\beta}):\tt M\to N$ \emph{derivation over} $f$ is a map in
$$\op{Der}_f({\tt M, N})=\varprojlim_{\beta}\varinjlim_{\alpha}\op{Der}_{f_{\alpha\beta}} ({\tt M}_\alpha,  {\tt N}_\beta).$$
\end{defi}
It is readily verified that the notion of a derivation over $f$ is independent of the choice of a representing family $(f_{\alpha,\beta})$.
Unsurprisingly, differential pro-graded $A$ modules are defined as pro-graded $\tt A$ modules together with a square-zero derivation:

\begin{defi}\label{dva}
Given a differential graded $k$-algebra $A=({\tt A},d_A)$, a \emph{pro-$A$-module} $M=({\tt M},d)$ is a pair  of a pro-graded $\tt A$-module $M$ and a square-zero derivation $d\in\op{Der}(\tt M)$. Morphisms of pro-$A$-modules are morphisms of underlying pro-graded $\tt A$-modules which intertwine the differential. The category of pro-$A$-modules is denoted by ${\tt ProMod}(A)$. Its full subcategory whose objects are pairs $({\tt M},d)$ such that $\tt M$ is pro finite is called the category of \emph{pro finite $A$-modules}, and denoted by ${\tt ProMod}_{\op{finite}}(A)$. Categories of \emph{ind-$A$-modules} and \emph{ind finite $A$-modules} are defined analogously, and denoted respectively by ${\tt IndMod}(A)$ and ${\tt IndMod}_{\op{finite}}(A)$.
\end{defi} 
Both the square-zero property of the differential, and the intertwining property of morphisms have concrete characterizations on the level of representing families:
\begin{prop}\label{concrete}
\begin{enumerate}
\item
Given a pro-graded $\tt A$-module ${\tt M}=({\tt M}_{\alpha})_{\alpha\in A}$, a derivation $d=(d_{\alpha,\beta})\in\op{Der}(\tt M)$, squares to zero if and only if for every $\beta\in A$ there exist $\alpha, \alpha'\in A$ such that $d_{\alpha',\beta}\circ d_{\alpha,\alpha'}=0$.
\item Given $M=(({\tt M}_{\alpha})_{\alpha\in A},d^M), N=(({\tt N}_{\beta})_{\beta\in B},d^N)\in{\tt ProMod}(A)$, a map $f=(f_{\alpha,\beta}):\tt M\to N$ intertwines the differentials if and only if for every $\beta\in B$, there exist $\beta'\in B$, $\alpha,\alpha'\in A$ such that $ f_{\alpha',\beta}\circ d^M_{\alpha,\alpha'}=d^N_{\beta',\beta}\circ f_{\alpha,\beta'}$.
\end{enumerate}
\end{prop}
\begin{proof}
Both statements follow immediately from the concrete realizations of diagrams in pro-categories, elaborated in the proof of \cite[Theorem 3.3.]{isa}.
\end{proof}

Duality of Lemma \ref{grduality} extends to the dg world:

\begin{prop}\label{dgduality} The full subcategory ${\tt Mod_{\op{flat}}} (A)\subset {\tt Mod }(A)$ of graded-flat $A$-modules is equivalent to the opposite of the category ${\tt ProMod}_{\op{finite}}(A)$ of pro finite $A$-modules. Equivalence and its inverse are both denoted by $(-)^*.$
\end{prop}

\begin{proof}
Let $M$ and $N$ be graded-flat $A$-modules, and let $f:M\to N$ be a morphism.
By Lazard's theorem, underlying graded $\tt A$-modules are colimits of directed systems of finite $\tt A$-modules $\tt M=\varinjlim_{\alpha} M_\alpha$, and $\tt N=\varinjlim_{\beta} N_\beta$. Let $(f_{\alpha,\beta}):(M_\alpha)\to( N_\beta)$ be the corresponding map of ind-finite $\tt A$-modules, in view of equivalence \ref{equiv}. Assume for now that
$$\op{Der}_f(\varinjlim_\alpha {\tt M}_\alpha, \varinjlim_\beta {\tt N}_\beta)\cong\varprojlim_{\alpha}\varinjlim_\beta\op{Der}_{f_{\alpha\beta}} ({\tt M}_\alpha,  {\tt N}_\beta).$$
The statement is proven later in the text. In fact, it is a special case of the isomorphism \ref{ref} for $l=1$.

Recall that dual of a finite $\tt A$-module $\tt M_{\alpha}$ is defined as the module of graded homomorphisms of $\tt A$-modules $\underline{\op{Hom}}_{\tt A}(\tt M_{\alpha},A).$ Explicitly, degree $k$ elements of $\underline{\op{Hom}}_{\tt A}(\tt M_{\alpha},A)$ are maps $f:{\tt M_{\alpha}}[k]\to \tt A$  of graded vector spaces such that $f(a\cdot x)=(-1)^{k|a|}a\cdot f(x)$, and the $\tt A$ module structure  is given by $(a\cdot f)(x)=a\cdot f(x)$. Functor $\underline{\op{Hom}}_{\tt A}(\tt -,A)$ is a contravariant self-equivalence of ${\tt Mod}_{\op{finite}}(\tt A)$, denoted by $(-)^*$.
Derivations between finite $\tt A$ modules also dualize. More precisely, there is a bijection
$$\op{Der}_{f_{\alpha\beta}} ({\tt M}_\alpha,  {\tt N}_\beta)\cong\op{Der}_{f_{\alpha\beta}^*} ({\tt N}_\beta^*,{\tt M}_\alpha^*)$$
which assigns to $X\in \op{Der}_{f_{\alpha\beta}} ({\tt M}_\alpha,  {\tt N}_\beta)$ a derivation given by
$$\phi\mapsto d_A\circ \phi\circ f_{\alpha\beta}-(-1)^{|\phi|}\phi\circ X.$$

Hence,
\begin{equation}\label{bitna}\op{Der}_f( {\tt M},{\tt N})\cong\varprojlim_{\alpha}\varinjlim_\beta\op{Der}_{f_{\alpha\beta}^*} ({\tt N}_\beta^*,{\tt M}_\alpha^*).\end{equation}

Setting $f:\tt M\to M$ to the identity, one concludes that differentials on $\tt M$ are in 1-1 correspondence with the differentials of $\tt M^*$, establishing the equivalence on objects (for the square-zero property, see Proposition \ref{concrete}). Given $M, N\in {\tt Mod}(A)$ and an $\tt A$-module morphism $f:\tt M\to N$,  derivations $d_N\circ f,f\circ d_M\in \op{Der}_{f}(\tt M, N)$ are equal if and only if the dual derivations $\varprojlim_{\alpha}\varinjlim_\beta\op{Der}_{f_{\alpha\beta}^*} ({\tt N}_\beta^*,  {\tt M}_\alpha^*)$ are equal. This establishes the equivalence on morphisms.
\end{proof}

I conclude the section with a rather obvious observation, which turns out to be quite useful:

\begin{prop}\label{adj}Given $A, B\in{\tt dgca}(k)$, and a morphism $f: A\to B$, the forgetful functor $\op{Pro}({\tt Mod}(B))\to \op{Pro}({\tt Mod}(A))$ has a left adjoint base-change functor 
$$B\otimes_A -:\op{Pro}({\tt Mod}(A))\to \op{Pro}({\tt Mod}(B)),\hspace{5pt}(M_{\alpha})\mapsto (B\otimes_A M_{\alpha}).$$
\end{prop}

Denote by ${\tt gca}_{\mathbb{Z}}(k)$ the category of $\mathbb{Z}$-graded commutative $k$-algebras. As with differential pro-graded modules, the correct notion of a differential pro-graded algebra involves a Leibniz rule relative to a morphism of graded algebras:
\begin{defi}\label{der}
Given ${\tt A,B}\in {\tt gca}_{\mathbb{Z}}(k)$ and a morphism $f:\tt A\to B$, we say that a graded $k$-linear map $d_f:\tt A\to B[-1]$ is a \emph{derivation over $f$}
if for all $x,y\in \tt A$
$$d_f(x\cdot y)=d_f(x)\cdot f(y) + (-1)^{|x|}f(x)\cdot d_f(y).$$
The set of all derivations over $f$ is denoted by $\op{Der}_f(\tt A, B).$ 
\end{defi}
The analogous statement to the Proposition \ref{prop2} holds in this setting, leading to the meaningful notion of a differential pro-graded commutative algebra:
\begin{defi}Given a pro-graded $k$-algebra ${\tt A}=({\tt A}_\alpha)\in \op{Pro}({\tt gca}_{\mathbb{Z}}(k))$, derivations of $\tt A$ are the following endomorphisms of pro-graded $k$-vector spaces:
$$\op{Der}(\tt M)=\varprojlim_{\beta}\varinjlim_{\alpha\geq \beta}\op{Der}_{\pi_{\alpha\beta}} ({\tt A}_\alpha,  {\tt A}_\beta)\subset \varprojlim_{\beta}\varinjlim_{\alpha\geq \beta}\op{Hom}_{k} ({\tt A}_\alpha,  {\tt A}_\beta[-1])=\op{Hom}_{\op{Pro}({\tt gMod}(k))}(\tt A, A[-1]).$$
More generally, given a map of pro-graded $\tt A$-algebras $f=(f_{\alpha,\beta}):\tt A\to B$, \emph{derivation over} $f$ is a map in
$$\op{Der}_f({\tt A, B})=\varprojlim_{\beta}\varinjlim_{\alpha}\op{Der}_{f_{\alpha\beta}} ({\tt A}_\alpha,  {\tt B}_\beta).$$
\end{defi}
Once again, it is readily verified that the notion of a derivation over $f$ is independent of the choice of a representing family $(f_{\alpha,\beta})$.
Finally, differential pro-graded $A$ algebras are defined as pro-graded $\tt A$ algebras together with a square-zero derivation:
\begin{defi}
\emph{Differential pro-graded commutative algebra} is a pair $({\tt A},d_A)$, of a pro-graded $k$-algebra and a square-zero derivation $d\in\op{Der}(\tt A)$. Morphisms of differential pro-graded commutative algebras are morphisms of underlying pro-graded algebras which intertwine the differential.
\end{defi}
Square-zero and intertwining properties are characterized as follows:
\begin{prop}\label{cocnrete2}
\begin{enumerate}
\item
Given a pro-graded $k$-algebra ${\tt A}=({\tt A}_{\alpha})_{\alpha\in A}$, a derivation $d=(d_{\alpha,\beta})\in\op{Der}(\tt A)$, squares to zero if and only if for every $\beta\in A$ there exist $\alpha, \alpha'\in A$ such that $d_{\alpha',\beta}\circ d_{\alpha,\alpha'}=0$.
\item Given $A=(({\tt A}_{\alpha})_{\alpha\in A},d^A),$ $B=(({\tt B}_{\beta})_{\beta\in B},d^B)\in{\tt ProMod}(A)$, a map $f=(f_{\alpha,\beta}):\tt A\to B$ intertwines the differentials if and only if for every $\beta\in B$, there exist $\beta'\in B$, $\alpha,\alpha'\in A$ such that $ f_{\alpha',\beta}\circ d^A_{\alpha,\alpha'}=d^B_{\beta',\beta}\circ f_{\alpha,\beta'}$.
\end{enumerate}
\end{prop}

\section{Pro-nilpotently extended cdga-s}

My reason to introduce differential pro-graded algebras is to deal with Chevalley-Eilenberg complexes of SH Lie Rinehart pairs. Let $(A,L)$ be Lie Rinehart pair, where $A$ is a commutative algebra over $k$, and $L$ a finitely generated projective $A$-module. Its Chevalley-Eilenberg complex is commonly interpreted as a dg algebra $(\op{Sym}_A(L^*[1]),d)$ over $A$. In the cohomological degree zero, the differential yields a derivation on $A$ valued in $L^*$ which is dual to the anchor, and in degree one a map $L^*\to L^*\wedge L^*$ dual to the bracket. Morphisms of such Lie-Rinehart pars are dual to maps of Chevalley-Eilenberg complexes $f:(\op{Sym}_A(L^*[1]),d)\to(\op{Sym}_B(M^*[1]),d)$, which themselves fit in a commutative square of dgca-s
\begin{center}
\begin{tikzcd}
(\op{Sym}_A(L^*[1]),d)\arrow[r,"f"]\arrow[d]&\arrow[d](\op{Sym}_B(M^*[1]),d)\\
A\arrow[r,"f_0"]&B.
\end{tikzcd}
\end{center}
First generalization is to allow for $L$ to be a perfect differential graded $A$-module concentrated in  non-positive degrees.  Any differential on $\op{Sym}_A(L^*[1])$ extending that on $L^*$ endows $L$ with the algebraic counterpart of Lie $n$-algebroid structure (\cite{PB}, \cite{PS},$\ldots$), which could be called Lie Rinehart $n$-pair. Morphisms of "Lie Rinehart $n$-pairs" are again encoded by the morphisms of dg algebras and they fit in the same commutative diagram.

If one either drops the non-positivity assumption on $L$, or allows $A$ to be a differential graded algebra (example of the later is the BV complex), or both, two new features appear: 
\begin{itemize}
\item
Firstly, $L$ can be endowed with a proper strong homotopy Lie Rinehart structure (\cite{Vit},$\ldots$), with infinite sequence of higher brackets and anchors. To encode this feature within the Chevalley-Eilenberg complex, one needs to consider differentials on the completion of the graded algebra $\op{Sym}_A(L^*[1])$.
\item
Secondly, not every differential on the completion will be intertwined by the projection $\widehat{\op{Sym}}_A(L^*[1])\to A$. In fact, only those differentials which satisfy the later property encode Chevalley-Eilenberg complexes of strong homotopy Lie-Rinehart pairs.
\end{itemize}

As formal power series are pro-nilpotent extensions of the base algebra, the appropriate category to  study Chevalley-Eilenberg complexes of SH Lie Rinehart pairs is that of pro-nilpotent extensions of differential graded algebras. This motivates the following definition:

\begin{defi}\emph{A pro-nilpotently extended dgca (fat cdga for short)} is a pair $(A,A_0)$, with $A_0\in {\tt dgca}(k)$, and $A$ a differential pro-graded commutative algebra, together with a morphism $\pi_A:A\to A_0$ of differential pro-graded commutative algebras represented by a compatible family of surjections $({\pi_A}_\alpha:A_\alpha\to A_0)$ with nilpotent kernels. A morphism of pro-nilpotently extended (fat) dgca-s $(A,A_0)$ and $(B,B_0)$ is a pair $(f,f_0)$, where $f:A\to B$ is a morphisms of differential pro-graded commutative algebras, and $f_0:A_0\to B_0$ is a morphism of dg algebras,  such that $\pi_B\circ f=f_0\circ \pi_A$. The category of fat dgca-s is denoted by ${\tt fcdga}(k)$.
\end{defi}

\section{Strong homotopy Lie Rinehart pairs and their Chevalley-Eilenberg complexes}

Duality between Lie algebroids and their Chevalley-Eilenberg complexes was first observed in the celebrated work of Vaintrob \cite{Vai}, and has seen numerous generalizations since, in particular to Lie $n$-algebroids (\cite{PB}, \cite{PS},...), and to SH Lie Rinehart pairs (\cite{Hueb},\cite{Vit}...). Later objects generalize dg Lie algebroids of Vezzosi and Nuiten (\cite{Vez}, \cite{Nui1}) by allowing both higher brackets and anchors. In all the above references the duality is either constructed in the smooth setting, or in the algebraic setting by restricting to "Lie algebroids" which are finitely generated and projective over the base (dg) algebra. Working with pro-objects, I relax the later finiteness condition asking only that the module underlying a SH Lie Rinehart pair be graded-flat over its base dg algebra. Conceptually, there is little novel in this section. It rather serves to show that, reasoning in line with the preceding section, generalization of the duality, as constructed in \cite{Vit}, is straightforward.

In accordance with the philosophy of "derivatives relative to an $\tt A$-module homomorphism", I introduce the relative version of Vitagliano's multiderivatives:

\begin{defi}\label{multider}
Let $l\in\mathbb{N}$. Given $\tt M, N\in {\tt Mod}(A)$, and an $\tt A$-module homomorphism $f:\tt M\to N$, a \emph{multiderivation over $f$ of weight $l$} is a pair $\mathbb{X}=(X,\sigma_X)$, where $\sigma_X$ is  a  map of $\tt A$-modules $\op{Sym}_{\tt A}^{l}({\tt M})\to \op{Der}({\tt A})[-1]$ (for $l=0$ equivalently a dergree 1 element in $ \op{Der}({\tt A})$), and $X$ is a map of graded vector spaces $\op{Sym}_k^{l+1}({\tt M})\to {\tt N}[-1]$ such that
$$X(m_1\odot\ldots\odot m_{l}\odot a\cdot m_{l+1})=\sigma_X (m_1\odot\ldots\odot m_{l})(a)\cdot f(m_{l+1})+(-1)^{|a|(1+\sum_{i=1}^l |m_i|)} a\cdot X(m_1\odot\ldots\odot m_{l+1}).$$
In the sequel, I refer to $X$ as multiderivation, and to $\sigma_X$ as its anchor. The set of multiderivations over $f$ of weight $l$ is denoted by $\op{mDer}_{f}^l(\tt M, N)$. The set of (inhomogeneous) multiderivations over $f$ is $\op{mDer}_{f}({\tt M, N}):=\prod_{l\in\mathbb{N}}\op{mDer}_{f}^l(\tt M, N)$.
\end{defi}

\begin{ex}\label{ex}
Let $f:\tt L \to M$, and  $g:\tt M\to  N$ be $\tt A$-module homomorphisms. If $(X,\sigma_X)$ is a multiderivation over $f$, $(g\circ X,\sigma_X)$ is a multiderivation over $g\circ f$ of the same weight. If $(Y,\sigma_Y)$ is a multiderivation over $g$ of weight $l$, $(Y\circ f^{\odot_k (l+1)},\sigma_Y\circ f^{\odot_{\tt A} l})$ is also a multiderivation over $g\circ f$.
Throughout the text, where there is no risk of confusion, I have relaxed the notation, denoting the above compositions by $g\circ X$ and $Y\circ f$. 
\end{ex}

The example is fundamental when defining multi-derivations of ind-modules.

Similarly to usual derivations, multiderivations of a free graded module are uniquely determined by their action on generators. The proof is straightforward, and hence omitted.
\begin{prop}\label{prop}
Let $\mathbb{V}$ be a graded $k$-vector space, ${\tt M=A}\otimes_k \mathbb{V}$ a free graded $\tt A$-module, $\tt N$ any graded $\tt A$-module, and $f:\tt M\to N$ a morphism of graded $\tt A$-modules. Morphism of graded vector spaces $X:\op{Sym}^{l+1}_k\mathbb{V}\to \tt N[-1]$, and a morphism of graded $\tt A$-modules $\sigma:\op{Sym}^{l}_{\tt A}\tt M\to \op{Der}(\tt A)[-1]$  uniquely determine a multiderivation over $f$ or weight $l$.
\end{prop}

Let $f:\tt L \to M$, and  $g:\tt M\to  N$ be $\tt A$-module homomorphisms. For  a multiderivation $\mathbb{X}$ of weight $k$ over $f$, and a multiderivation  $\mathbb{Y}$ of weight $l$ over $g$, we say that
$\mathbb{Y}\circ\mathbb{X}=0$ if 
\begin{equation*}
\begin{split}
(X\circ Y)&(m_1\odot\ldots \odot m_{k+l+1})\\
:=&\sum_{\sigma\in \op{Sh}(l+1,k)}(-1)^{|m_1,\ldots m_{k+l+1}|_{\sigma}}X(Y( m_{\sigma (1)}\odot\ldots\odot m_{\sigma (l+1)})\odot m_{\sigma (l+2)}\odot\ldots\odot m_{\sigma (k+l+1)})=0,
\end{split}
\end{equation*}
and
\begin{equation*}
\begin{split}
(\sigma_{X}\circ& Y+\sigma_{X}\sigma_{Y})(m_1\odot\ldots \odot m_{k+l})\\
:=&\sum_{\sigma\in \op{Sh}(l+1,k-1)}(-1)^{|m_1,\ldots m_{k+l}|_{\sigma}}\sigma_{X}(Y( m_{\sigma (1)}\odot\ldots\odot m_{\sigma (l+1)})\odot m_{\sigma (l+2)}\odot\ldots\odot m_{\sigma (k+l)})\\
+&\sum_{\sigma\in \op{Sh}(l,k)}(-1)^{|m_1,\ldots m_{k+l}|_{\sigma}}(-1)^{m_{\sigma(1)}+\ldots + m_{\sigma(l)}}\sigma_{X}( m_{\sigma (1)}\odot\ldots\odot m_{\sigma (l)})\sigma_{Y}( m_{\sigma (l+1)}\odot\ldots\odot m_{\sigma (l+k)})\\
=&0
\end{split}
\end{equation*}
A multiderivation $\mathbb{X}$ over identity squares to zero if $\mathbb{X}\circ\mathbb{X}=0 $.

Given (inhomogeneous) multiderivations $\mathbb{X}=\prod_{k\in \mathbb{N}} \mathbb{X}_k$ over $f$, and $\mathbb{Y}=\prod_{k\in \mathbb{N}} \mathbb{Y}_k$ over $g$, we say that $\mathbb{Y}\circ\mathbb{X}=0$ if for all $k\in \mathbb{N}$,
$$\sum_{i+j=k}\mathbb{Y}_i\circ \mathbb{X}_j=0.$$

\begin{ex}
A square-zero multiderivation of weight zero $\mathbb{X}=(X,\sigma_X)$ over $\op{id}_{\tt M}$ is equivalently a dg module $({\tt M},X)$ over the dg algebra $({\tt A},\sigma_X)$.
\end{ex}

Let $(\tt A,\tt L)$ be a pair of an $\mathbb{N}$-graded commutative (base) algebra $\tt A$, and a  $\mathbb{Z}$-graded $\tt A$ module $\tt L$, together with a square-zero (inhomogeneous)  multiderivation $\mathbb{X}\in  \op{mDer}_{\op{id}}(\tt L, L)$. In particular, $\mathbb{X}_0$ squares to zero, hence $A=({\tt A},X_0)$ is a differential graded commutative algebra $A$, and $M=({\tt M},\sigma_{X_0})$ is a dg module over $A$. 

\begin{defi}(\cite[Definition 14]{Vit}) 
A SH Lie Rinehart pair is a pair $(A,L)$, with $A\in{\tt dgca}(k)$, and $L$ a $\mathbb{Z}$-graded $A$ module, together with an (inhomogeneous) square-zero multiderivation $\mathbb{X}\in  \op{mDer}_{\op{id}}({\tt L}, {\tt L})$ for which $L=({\tt L}[1],X_0)$, $A=({\tt A}, \sigma_{X_0})$.
\end{defi}
The paper does not address morphisms between SH Lie Rinehart pairs. Frankly, that definition would be a mess. However, by the Corollary \ref{cor}, SH Lie Rinehart pairs whose underlying graded modules are flat over its base are equivalently encoded within their Chevalley-Eilenberg complexes, which themselves form a full subcategory ${\tt CE}(k)$ of ${\tt fcdga}(k)$. To avoid technical complications, I define the category of flat SH Lie Rinehart pairs as the category opposite to ${\tt CE}(k)$.

The next goal is to give a precise meaning to the category ${\tt CE}(k)$.
Let $\tt A$ be a non-positively graded commutative algebra, and let ${\tt M}=({\tt M}_{\alpha},p_{\alpha,\beta})\in \op{Pro}(\tt Mod({\tt A}))$ be a pro-graded $\tt A$-module. 

Denote ${\tt M}_{(\alpha,k)}:=\op{Sym}_{\tt A}({\tt M}_{\alpha})/\op{Sym}^{>k}_{\tt A}({\tt M}_{\alpha})$, and for $\beta<\alpha$, $l<k$, denote by $p_{(\alpha,k)(\beta,l)}$ the composition
$$\op{Sym}_{\tt A}({\tt M}_{\alpha})/\op{Sym}^{>k}_{\tt A}({\tt M}_{\alpha})\xrightarrow{\op{Sym}_A p_{\alpha,\beta}} \op{Sym}_{\tt A}({\tt M}_{\beta})/\op{Sym}^{>k}_{\tt A}({\tt M}_{\beta})\twoheadrightarrow \op{Sym}_{\tt A}({\tt M}_{\beta})/\op{Sym}^{>l}_{\tt A}({\tt M}_{\beta}).$$

Inverse system $({\tt M}_{(\alpha,k)},p_{(\alpha,k)(\beta,l)})$ of nilpotent extensions of $\tt A$ is denoted by $\widehat{\op{Sym}}_{\tt A}(\tt M)$, and the canonical projection to $\tt A$ is denoted by $\pi_{\tt M}:\widehat{\op{Sym}}_{\tt A}(\tt M)\to \tt A$.

\begin{defi} Chevalley-Eilenberg category, denoted by ${\tt CE} (k)$, is the full subcategory of ${\tt fcdga}(k)$ whose objects are pairs $((\widehat{\op{Sym}}_{\tt A}({\tt M}),d),{ A})$, with $A\in{\tt cdga}(k)$, $\tt M$ a pro-finite $\tt A$-module, and $d$ any differential on $\widehat{\op{Sym}}_{\tt A}(\tt M)$ which is intertwined by $\pi_{\tt M}$ with the differential on $A$.
\end{defi}

To establish a duality between the objects of ${\tt CE} (k)$ and SH Lie Rinehart pairs, I first introduce the weight decomposition on the Chevalley-Eilenberg side.

A morphism of pro-$\tt A$-modules $f=(f_{\alpha,\beta}):\tt M\to N$ induces a morphism of pro-nilpotent algebras $f_{(\alpha,k)(\beta,l)}: \widehat{\op{Sym}}_{\tt A}(\tt M)\to \widehat{\op{Sym}}_{\tt A}(\tt N)$. A derivation $d=(d_{(\alpha,k)(\beta,l)}):\widehat{\op{Sym}}_{\tt A}(\tt M)\to \widehat{\op{Sym}}_{\tt A}(\tt N)[-1]$ over $f$
decomposes by weights. Namely, for $i\leq k$ and, the weight $n$ component of ${d}_{(\alpha,k)(\beta,l)}$ is a derivation which assigns to $x\in \op{Sym}^i_{\tt A}({\tt M}_{\alpha})$ the homogeneous component of ${d}_{(\alpha,k)(\beta,l)}(x)$ in $\op{Sym}^{i+n}_{\tt A}({\tt N}_{\beta})$ (zero if $i+n>l$). The set of homogeneous derivations over $f_{(\alpha,k)(\beta,l)}$ of weight $n$ is denoted by $\op{Der}^n_{f_{(\alpha, k)(\beta,l)}} ({{\tt M}}_{\alpha,k},{{\tt N}}_{\beta, l})$. As weight components are preserved under the composition with both $p_{(\alpha,k)(\alpha',k')}$ and $p_{(\beta,l)(\beta',l')}$,  they induce a weight decomposition of $d$, denoted by $d=\prod_{n\geq -1} d^n$. Assuming that $\pi_{\tt N}\circ d$ factors through $\pi_{\tt M}$, the $(-1)$ weight component is zero.  

In particular, differential $d$ of a Chevalley-Eilenberg complex $((\widehat{\op{Sym}}_{\tt A}({\tt M}),d),{ A})$ admits a weight decomposition in non-negative weights. Its weight zero component recovers the differential of $A$, and determines a differential $d^l$ on $\tt M$, called the \emph{linear part} of $d$, which respects the Leibniz rule for $\tt A$-action. Hence, $M=({\tt M},d^l)\in{\tt ProMod}_{\op{finite}}(A)$. Similarly, any morphism of pro-graded commutative algebras $g: \widehat{\op{Sym}}_{\tt A}(\tt M)\to \widehat{\op{Sym}}_{\tt B}(\tt N)$ admits a weight decomposition, in particular, any morphism of Chevalley-Eilenberg complexes. Again, morphisms of Chevalley-Eilenberg complexes are concentrated in non-negative weights. Moreover,  its weight zero component intertwines the weight zero of the differential, hence determines a map $M\to N$  in ${\tt ProMod}(A)$, equivalently  (in light of the adjunction \ref{adj}) a map  $g^l:B\otimes_A M\to N$ in ${\tt ProMod}_{\op{finite}}(B)$. $g^l$ is called the \emph{linear part} of $g$.


\begin{theo}\label{Theo1}
Let $\tt A$ be a non-positively graded graded-commutative algebra, $\tt M, N \in \tt Mod_{\op{flat}}(A)$, $f:\tt M\to N$. Let $f^*:\tt N^* \to M^*$ be the morphism in $\op{Pro}({\tt Mod}_{\op{finite}}(\tt A))$ dual to $f$. There is a bijection
\begin{equation}\label{bitna}\op{mDer}_f( {\tt M},{\tt N})\cong \op{Der}^{\geq 0}_{f^*}(\widehat{\op{Sym}}_{\tt A}{\tt N^*},\widehat{\op{Sym}}_{\tt A}{\tt M^*}).\end{equation}
Composition of multiderivatives (on the left) squares to zero if and only if the composition of corresponding derivatives (on the right) squares to zero.
\end{theo}

\begin{proof}
As multiderivations over $f:\tt M\to N$ and derivations over $f^*:\widehat{\op{Sym}}_{\tt A}(\tt N^*)\to \widehat{\op{Sym}}_{\tt A}(\tt M^*)$ both admit weight decompositions, it suffices to prove duality for a fixed weight (see Equation \ref{bitna}). That the two square-zero properties are equivalent follows from the explicit construction of the equivalence, and the Proposition \ref{cocnrete2}. 

Fix a weight $l\in\mathbb{N}$. Let ${\tt M}=\varinjlim_{\alpha\in A} {\tt M}_\alpha$, and ${\tt N}=\varinjlim_{\beta\in B} {\tt N}_\beta$, with $\tt M_\alpha, N_\beta\in  \tt Mod_{\op{finite}}(A)$. For $\alpha<\alpha'\in A$, denote by $i_{\alpha\alpha'}:\tt M_{\alpha}\to M_{\alpha'}$ maps of the injective system $({\tt M}_{\alpha})$ and by 
$i_\alpha: \tt M_{\alpha}\to M$ maps into the colimit. Observe that $\op{Sym}^{l+1}_k({\tt M})=\varinjlim_{\alpha}\op{Sym}^{l+1}_k(\tt M_{\alpha})$.
Let $f:\tt M \to N$ be an $\tt A$-module map. 

In the bijection
\begin{equation}\label{bzvz1}\op{Hom}_{{\tt Mod} (k)}( {\op{Sym}^{l+1}_k(\tt M)}, {\tt N}[-1])\cong\varprojlim_\alpha\op{Hom}_{{\tt Mod} (k)}( \op{Sym}^{l+1}_k(\tt M_{\alpha}), {\tt N}[-1]),\end{equation}
 whenever $X$ is a multiderivation over $f$ with anchor $\sigma$, $X\circ i_{\alpha}^{\otimes_{\tt A} l+1}$ is a multiderivation over $f\circ i_\alpha$ (see the Example \ref{ex}) with anchor $\sigma_{\alpha}:=\sigma\circ {i_{\alpha}}^{\otimes_{\tt A} l}$. Clearly, for $\alpha<\alpha'$, $\sigma_{\alpha}=\sigma_{\alpha'}\circ {i_{\alpha\alpha'}}^{\otimes_{\tt A} l}$, and the bijection
\begin{equation}\label{bzvz1.5}\op{Hom}_{{\tt Mod} ({\tt A})}( {\op{Sym}^{l}_{\tt A}(\tt M)}, (\op{Der} {\tt A})[-1])\cong\varprojlim_\alpha\op{Hom}_{{\tt Mod} ({\tt A})}( \op{Sym}^{l}_{\tt A}(\tt M_{\alpha}),(\op{Der} {\tt A})[-1])\end{equation}
maps $\sigma$ to $(\sigma_\alpha)_{\alpha\in I}$.
Conversely, given $g\in\op{Hom}_{{\tt Mod} (k)}( {\op{Sym}^{l+1}_k(\tt M)}, {\tt N}[-1])$, if for each $\alpha$, $g\circ i_\alpha^{\otimes_{\tt A} l+1}$ is a multiderivation over $f\circ i_\alpha$, whose anchors satisfy $\sigma_{\alpha'}=\sigma_{\alpha}\circ {i_{\alpha\alpha'}}^{\otimes_{\tt A} l}$ for $\alpha<\alpha'$, then $g$ is a multiderivation over $f$ whose anchor $\sigma$ corresponds under the bijection \ref{bzvz1.5} to the directed system $(\sigma_{\alpha})$. Namely, every monomial $m_1\odot\ldots\odot m_{l+1}\in \op{Sym}^{l+1}_k(\tt M)$ is in the image of $i_\alpha$ for some $\alpha$, hence
\begin{equation*}
\begin{split}
g(m_1\odot\ldots\odot m_{l}\odot a\cdot m_{l+1})=&(g\circ i_{\alpha}^{\otimes_{\tt A} l+1})(m_1^{\alpha}\odot\ldots\odot m_{l}^{\alpha}\odot a\cdot m_{l+1}^{\alpha})\\
=&\sigma_\alpha (m_1^{\alpha}\otimes\ldots\otimes m_{l}^{\alpha})(a)\cdot f\circ i_\alpha(m_{l+1}^{\alpha})\\
+&(-1)^{|a|(1+\sum |m_i|)} a\cdot g\circ i_\alpha^{\otimes_{\tt A} l+1}(m_1^{\alpha}\otimes\ldots\otimes m_{l+1}^{\alpha})\\
=&\sigma (m_1\otimes\ldots\otimes m_{l})(a)\cdot f(m_{l+1})+(-1)^{|a|(1+\sum |m_i|)} a\cdot g(m_1\otimes\ldots\otimes m_{l+1}).
\end{split}
\end{equation*}
Finally,
\begin{equation}\label{bzvz2}\op{mDer}^l_f(\varinjlim_\alpha {\tt M}_\alpha, \varinjlim_\beta {\tt N}_\beta)\cong\varprojlim_\alpha\op{mDer}^l_{f\circ i_\alpha}( {\tt M}_\alpha, \varinjlim_\beta {\tt N}_\beta).\end{equation}

Denote
$$f=\{f_{\alpha \beta}\}\in\varprojlim_\alpha\varinjlim_\beta\op{Hom}_{\tt Mod (A)} ({\tt M}_\alpha,  {\tt N}_\beta).$$
 For a fixed $\alpha$, and any $\beta$ for which  $f_{\alpha \beta}$ is in the compatible family, $f\circ i_\alpha=i_\beta \circ f_{\alpha\beta}$.
The map
\begin{equation}\label{bzvz3}
 \varinjlim_\beta\op{mDer}^l_{f_{\alpha\beta}} ({\tt M}_\alpha,  {\tt N}_\beta)\to\op{mDer}^l_{f\circ i_\alpha}( {\tt M}_\alpha, \varinjlim_\beta {\tt N}_\beta) ,\hspace{5pt} [(X_{\alpha\beta},\sigma_\alpha)]\mapsto (i_\beta\circ X_{\alpha\beta}, \sigma_\alpha)\end{equation}
is bijective. Indeed, consider any $\beta''$ with $f_{\alpha\beta''}$ in the compatible family, and denote by $\{x_1\ldots x_n\}$ a basis for $\tt M_{\alpha}$. Given $(X_{\alpha},\sigma_\alpha)\in \op{mDer}^l_{f\circ i_\alpha}( {\tt M}_\alpha, \varinjlim_\beta {\tt N}_\beta)$, let $\beta'\geq \beta''$ be the smallest index for which all $X_{\alpha}(x_{i_1}\odot\ldots\odot x_{i_{l+1}})$ lie in the image of $i_{\beta'}.$ Let $x_{(i_1,\ldots,i_{l+1})}^{\beta'}\in\tt N_{\beta'}$ be any element in the $i_{\beta'}$-preimage of $X_{\alpha}(x_{i_1}\odot\ldots\odot x_{i_{l+1}})$. It follows that there is a unique multiderivation $X_{\alpha\beta'}$ over $f_{\alpha\beta'}$ with anchor $\sigma_\alpha$ such that $X_{\alpha\beta'}(x_{i_1}\odot\ldots\odot x_{i_{l+1}})=x_{(i_1,\ldots,i_{l+1})}^{\beta'}$. By the proposition \ref{prop},  $i_{\beta'}\circ X_{\alpha\beta'}=X_\alpha$, showing that the map (\ref{bzvz3}) is indeed bijective. Finally, \ref{bzvz2} and \ref{bzvz3} give
\begin{equation}\label{ref}
\op{mDer}_f^l(\varinjlim_\alpha {\tt M}_\alpha, \varinjlim_\beta {\tt N}_\beta)\cong\varprojlim_{\alpha}\varinjlim_\beta\op{mDer}^l_{f_{\alpha\beta}} ({\tt M}_\alpha,  {\tt N}_\beta).\end{equation}

Take $(\mathbb{X}_{\alpha,\beta})\in\varprojlim_{\alpha}\varinjlim_\beta\op{mDer}^l_{f_{\alpha\beta}} ({\tt M}_\alpha,  {\tt N}_\beta).$

Each $\mathbb{X}_{\alpha\beta}$ induces a family of derivations $\{d_{(\beta,n)(\alpha,m)}\in \op{Der}^{l}_{f_{\alpha\beta}^*} ({\tt N}_{\beta,n}^*,{\tt M}_{\alpha,m}^*):m\leq n+l\}$, defined by $d_{(\alpha,m)(\beta,n)}(\omega)=0$, for $\omega\in \op{Sym}^k_{\tt A} \tt N_{\beta}^*$, with $k+l>m$, and otherwise by

\begin{equation*}
\begin{split}
&(d_{(\beta,n)(\alpha,m)}\omega)(x_1\odot\ldots\odot x_{k+l})=\\
& \sum_{\sigma\in \op{Sh}(l,k)}(-1)^{|\omega|(x_{\sigma(1)}+\ldots x_{\sigma(l)})+|x_1,\ldots,x_{k+l}|_\sigma}\sigma_{X_{\alpha\beta}}(x_{\sigma(1)}\odot\ldots \odot x_{\sigma(l)})(\omega(f(x_{\sigma(l+1)})\odot\ldots \odot f(x_{\sigma(k+l)})))\\
&-(-1)^{|\omega|}\sum_{\sigma\in \op{Sh}(l+1,k-1)}(-1)^{|x_1,\ldots x_{k+l}|_{\sigma}}\omega(X( x_{\sigma (1)}\odot\ldots\odot x_{\sigma (l+1)})\odot f(x_{\sigma (l+2)})\odot\ldots\odot f(x_{\sigma (k+l)})).
\end{split}
\end{equation*}

It is an exercise to show that the family $( d_{(\beta,n)(\alpha,m)})$ is compatible.  Hence,
$$(d_{(\beta,n)(\alpha,m)})\in\varprojlim_{(\alpha,m)}\varinjlim_{(\beta,n)}\op{Der}^{l}_{f_{(\alpha, m)(\beta,n)}^*} ({{\tt N}^*}_{\beta,n},{{\tt M}^*}_{\alpha,m}).$$

Conversely, starting from a compatible family $(d_{(\beta,n)(\alpha,m)})$, $(X_{\alpha,\beta})$ is reconstructed as follows. For fixed index $\alpha$ let $\beta$ and $n\geq 1$ be such that $d_{(\beta,n)(\alpha,l)}$ is contained in the compatible family, and define the multiderivation $X_{\alpha \beta}$ implicitly for $x^*\in \tt N_\beta^*$ by
\begin{equation*} 
\begin{split}
x^*(X&_{\alpha,\beta}(x_1\odot\ldots\odot x_{l+1}))\\
=&\sum_{i=1\ldots l+1} (-1)^{|x^*|(1+|x_1|+\ldots|x_{l+1}|-|x_i|)+|x_i|(|x_{i+1}+\ldots + |x_{l+1}|)}(\sigma_{X_{\alpha \beta}}(x_1\odot\ldots\hat{x}_i\ldots \odot x_{l+1}))(x^*(f(x_i)))\\
+&(-1)^{|x^*|}(d_{(\beta,n)(\alpha,l)})(x^*)(x_1\odot\ldots\odot x_{l+1})
\end{split}
\end{equation*}
with the anchor is given by
$$(\sigma_{X_{\alpha \beta}}(x_1\odot\ldots\odot x_{l}))(a)=(-1)^{|a|(|x_1|+\ldots+|x_l|)}(d_{(\beta,n)(\alpha,l)}(a))(x_1\odot\ldots\odot x_{l}).$$
Finally,
\begin{equation}\label{bitna}\op{mDer}_f^l( {\tt M},{\tt N})\cong\varprojlim_{(\alpha,m)}\varinjlim_{(\beta,n)}\op{Der}_{f_{(\alpha,m)(\beta,n)}^*}^{l} ({\tt N}_{\beta,n}^*,{\tt M}_{\alpha,m}^*).\end{equation}
\end{proof}

Applying the theorem to $f=\op{id}$ we obtain the following:
\begin{cor}\label{cor}
For a graded-flat $A$-module $L$, the structure of a SH Lie Rinehart pair $(A,L)$ is equivalently encoded by a Chevalley-Eilenberg complex $((\widehat{\op{Sym}}_{\tt A}({\tt L}^*[1]),d),A)$ such that $({\tt L}^*,d^l)$ is dual to $L$.
\end{cor}

\section{Category of cofibrant objects}

Denote by  ${\tt CE}_{\op{cof}}(k)$ the full subcategory of ${\tt CE}(k)$ consisting of Chelvalley-Eilenberg complexes corresponding to SH Lie Rinehart pairs $(A,M)$ where $A$ is a semi-free $dg$-algebra over $k$, and $M$ is a cell complex in the model category of dg $A$-modules. Explicitly, a cell complex is a graded free $A$-module with a well-ordered set of homogeneous generators $(m_i)_{i\in I}$ whose differential satisfies the lowering condition
$$d(m_i)\in A\langle m_j\rangle_{j<i}.$$
A proof in a slightly more general setting of $\mathcal{D}$-modules is in \cite{hac}.

Take $A\in {\tt dgca}(k)$, and a free graded $\tt A$-module $\tt M={\tt A}\langle m_i\rangle_{i\in I}$. Denoting by $\vec{I}$ the small category of finite subsets of $I$ and inclusions,
$$ {\tt M}={\tt A}\langle m_i\rangle_{i\in I}=\varinjlim_{J\in \vec{I}} {\tt A}\langle m_j\rangle_{j\in J}.$$
Its dual pro-finite module is represented by the inverse system of projections ${\tt A}\langle m_j^*\rangle_{j\in J}\to {\tt A}\langle m_{j'}^*\rangle_{j'\in J'}$ for $J'\subset J$ whose projective limit is the graded $\tt A$ module $\prod_{i\in I}\tt A\langle m_i^* \rangle$, denoted by ${\tt A}\langle\langle m_i^*\rangle\rangle_{i\in I}$. Differential on the inverse system induces a differential on the projective limit.  Denoting
\begin{equation}\label{condition}dm_i^*=\prod_j a^j_i m_j^*,\end{equation}
for a fixed $j$ there are at most finitely many indices $i$ with $a^i_j\neq 0$.
 Conversely, there exists a unique differential on ${\tt A}\langle\langle m_i^*\rangle\rangle_{i\in I}$ which lifts to  a differential of pro-graded modules and satisfies (\ref{condition}). Its lift is the dual of an $A$-cell if and only if the differential satisfies rising condition $a^i_j=0$ for $i\leq j$. An important observation is that when $M$ is itself a cell complex in $A$-modules, the projective limit of its dual together with the induced differential is equal to the naive dual $\underline{\op{Hom}}_{{\tt Mod}(A)}(M,A)$. 

The same holds for morphisms. Given a ${\tt dgca}(k)$-map $f_0: A\to B$, there exists a unique morphism $f^*:{\tt A}\langle\langle m_i^*\rangle\rangle_{i\in I}\to {\tt B}\langle\langle n_j^*\rangle\rangle_{j\in J}$ of $\tt A$-modules which lifts to a morphism of pro-graded $\tt A$-modules and satisfies
\begin{equation}\label{conditionm}f^*(m_i^*)=\prod_j b^j_i n_j^*,\end{equation}
where for a fixed $j$ there are at most finitely many indices $i$ with $a^i_j\neq 0$. The map intertwines the differential if and only if its lift does.

Projective limit of the inverse system of graded-commutative algebras $\widehat{\op{Sym}}_{\tt A}({\tt M^*})$
is denoted by ${\tt A}[[m_i^*]]_{i\in I}$. Its elements are
$$y=\prod_{\substack{n\in \mathbb{N}\\i_1\ldots i_n\in I}}a^{i_1\ldots i_n}m_{i_1}^*\cdots m_{i_n}^*.$$
Similarly as above, a differential on ${\tt A}[[m_i^*]]_{i\in I}$ which lifts to the level of pro-graded algebras  is uniquely determined by
$$d_{M^*}|_A\in\op{Der}(A,A[[m_i^*]]),\hspace{10pt} d_{M^*}m_i^*=\prod_{\substack{n\in \mathbb{N}_{>0}\\i_1\ldots i_n\in I}}a_i^{i_1\ldots i_n}m_{i_1}^*\cdots m_{i_n}^*,$$
where for any n-tuple $(i_1\ldots i_n)\in I^{\times n}$ there is at most finitely many $i\in I$ for which the coefficient $a_i^{i_1\ldots i_n}$ is non-zero. If $d_{M^*}^2=0$, the lift is canonically a fat dgca $((\widehat{\op{Sym}}_{\tt A}({\tt M^*}),d_{M^*}),({\tt A},d_A))$, for $d_A=\pi_A\circ d_{M^*}|_A$. Linear part $d_{M^*}^l:\tt M^*\to M^*[-1]$ of $d_{M^*}$ is defined by the same equation as \ref{condition}.

A morphism $f^*:{\tt A}[[m_i^*]]_{i\in I}\to{\tt B}[[n_j^*]]_{j\in J}$ which lifts to the level of pro-graded algebras is uniquely determined by
$$f^*|_{\tt A}\in\op{Hom}_{{\tt Mod}({\tt A})}({\tt A},{\tt B}[[n_j^*]]),\hspace{10pt} f^*(m_i^*)=\prod_{\substack{n\in \mathbb{N}_{>0}\\j_1\ldots j_n\in J}}b_i^{j_1\ldots j_n}n_{j_1}^*\cdots n_{j_n}^*,$$
where for any n-tuple $(j_1\ldots j_n)\in J^{\times n}$ there is at most finitely many $i\in I$ for which the coefficient $b_i^{j_1\ldots j_n}$ is non-zero. Given fat cdga-s
 $$((\widehat{\op{Sym}}_{\tt A}({\tt M^*}),d_{M^*}),({\tt A},d_A)),\hspace{10pt}((\widehat{\op{Sym}}_{\tt B}({\tt N^*}),d_{N^*}),({\tt B},d_B)),$$
such a morphism $f^*$, if it intertwines differentials $d_{M^*}$ and $d_{N^*}$, determines canonically a morphism of fat cdga-s $(f^*,f_0)$, with $f_0=\pi_B\circ f^*|_A$. Its linear part ${f^*}^l:B\otimes_A M^*\to N^*$ is defined by the same equation as \ref{conditionm}.

\begin{rem}The above discussion enables one to work equivalently with limits of pro-graded objects together with appropriate differentials and morphisms, which in turn significantly simplifies calculations. In the sequel, I distinguish between pro-objects and their limits only when I deem it necessary for the understanding of the text.\end{rem}\bigskip

%

In order to endow the category ${\tt CE}_{\op{cof}}(k)$ with the structure of the category of cofibrant objects, the first step is to define weak equivalences and cofibrations.

A map in ${\tt CE}_{\op{cof}}(k)$ (with the differentials suppressed from the notation)  \begin{equation}(\label{map}f^*,f_0):(\widehat{\op{Sym}}_{\tt A}(\tt M^*),{\tt A})\to(\widehat{\op{Sym}}_{\tt B}(\tt N^*),{\tt B})\end{equation}
is a weak equivalence if both $f_0:A\to B$ and the dual $f^l:N\to B\otimes_A M$ of ${f^*}^l$ are weak equivalences.

Cofibrations are are maps (\ref{map}) such that
\begin{itemize}
\item
$f_0$ is a relative cell complex in ${\tt dgca}(k),$ that is, an inclusion of underlying free graded commutative algebras $k[x_a]_{a\in A}\hookrightarrow k[x_a,y_b]_{a\in A, b\in B};$
\item denoting ${\tt N}^*={\tt B}\langle\langle m_i^*\rangle\rangle_{i\in I}$, there exists
 a subset $I'\subset I$, such that ${\tt M}^*={\tt A}\langle\langle m_i^*\rangle\rangle_{i\in I'}$, and the map
 $f^*:{\tt B} \otimes_{\tt A} \widehat{\op{Sym}}_{\tt A}({\tt M^*})\to\widehat{\op{Sym}}_{\tt B}({\tt N^*})$ is the inclusion
 ${\tt B}[[m_i^*]]_{i\in I'}\to{\tt B}[[m_i^*]]_{i\in I}$ of formal power series. No lowering condition is required for the differential.
\end{itemize}

\subsection*{Finite coproducts.}
Coproduct of $({\tt A}[[m_i^*]]_{i\in I}, d_{M^*})$ and $({\tt B}[[n_j^*]]_{j\in J}, d_{N^*})$ is 
 the pair $$(({\tt A}\otimes_k{\tt B})[[m_i^*,n_j^*]]_{i\in I, j\in J}, d_{\sqcup}),$$
with 
\begin{equation*}
\begin{split}
&d_{\sqcup}(a\otimes b)=d_{M^*}(a)\cdot b+(-1)^{|a|}a\cdot d_{N^*}(b)\text{ for }a\in{\tt A}, b\in{\tt B};\\
&d_{\sqcup}(m_i^*)=d_{M^*}(m_i^*),\hspace{5pt}d_{\sqcup}(n_j^*)=d_{N^*}(n_j^*).
\end{split}
\end{equation*} 
$d^l_{\sqcup}$ satisfies the rising condition for the unique ordering of $I\cup J$ which is preserved by inclusions of both $I$ and $J$ and such that $i<j$ for all $i\in I, i\in J$.

Initial object is $(k,k)$ with zero-differential.

\subsection*{Pushouts along cofibrations}

Pushout in ${\tt CE}_{\op{cof}}(k)$ of a map $f^*:({\tt A}[[m_i^*]]_{i\in I'}, d_{M'})\to ({\tt C}[[n_j^*]]_{i\in J},d_N)$ along a cofibration $g^*:({\tt A}[[m_i^*]]_{i\in I'}, d_{M'})\to ({\tt A}[x_a]_{a\in A}[[m_i^*]]_{i\in I}, d_M)$ is on the level of pro-graded algebras given by the commutative diagram
\begin{center}
\begin{equation}\label{com1}
\begin{tikzcd}
{\tt A}[[m_i^*]]_{i\in I'}\arrow[r,"f^*"]\arrow[d,"g^*"]&{\tt C}[[n_j^*]]_{j\in J}\arrow[d,"\gamma^*"]\\
{\tt A}[x_a]_{a\in A}[[m_i^*]]_{i\in I}\arrow[r,"\varphi^*"]&{\tt C}[x_a]_{a\in A}[[m_i^*,n_j^*]]_{i\in I\setminus I', j\in J}.
\end{tikzcd}
\end{equation}
\end{center}
Set $(I\setminus I')\cup J$ is well ordered with the unique ordering which respects the inclusions from $(I\setminus I')$ and $J$, and such that $i<j$ for all $j\in J, i\in I\setminus I'$. The map $\gamma^*$ is defined as follows: on $\tt C$ it is inclusion ${\tt C}\hookrightarrow {\tt C}[x_a]_{a\in A}$, and it maps the formal generators to themselves. $\varphi^*$  is defined on ${\tt A}$ as the composition $\gamma^*\circ f^*$, polynomial generators $x_a$ are mapped to themselves, whereas on formal generators
$$
   \phi^*(m_i^*) =\begin{cases}
\gamma^*\circ f (m_i^*), & \text{ if } i\in I',\\
m_i^*, & \text{otherwise}.
\end{cases}
 $$
Differential $d$ on ${\tt A}[x_a]_{a\in A}\otimes_{\tt A} {\tt C}[[m_i^*,n_j^*]]_{i\in I\setminus I', j\in J}$ is the unique differential intertwined by $\gamma^*$ and $\varphi^*$. 
It is clear from the above defined ordering on $(I\setminus I')\cup J$ that the rising condition on differential's linear part is satisfied. Notice that $\gamma_0$ is the pushout in ${\tt dgca}(k)$ of the cofibration $g_0$ along $f_0$, hence a cofibration. With this, it is obvious that $\gamma^*$ is also a cofibration. It is left to show that $\gamma^*$ is a weak equivalence if $g^*$ is such. Again, in this case, $\gamma_0$ is a pushout of acyclic cofibration, hence acyclic. It remains to prove that $ \gamma^l$  is a weak equivalence, equivalently, that its kernel $\op{Ker}(\gamma^l)$ is acyclic. Denoting by $B$ the base dg algebra of the Chevalley-Eilenberg complex in the lower-left vertex of the commutative diagram \ref{com1} (${\tt B}:={\tt A}[x_a]_{a\in A}$), and by $D$ that of its lower-right vertex, it follows from the explicit construction that $\op{Ker}(\gamma^l)=D\otimes_B\op{Ker}(g^l)$. As a graded $\tt B$-module, $\op{Ker}(g^l)={\tt B}\langle m_i\rangle_{i\in I\setminus I'}$, and its differential satisfies the lowering condition with respect to the ordering inherited from $I$. Thus it is cofibrant and acyclic $B$-module, so its tensor product with $D$ is acyclic.
\subsection*{2-out-of-3} Let (with differentials once again suppressed from the notation)
\begin{equation*}
\begin{split}
 (f^*,f_0):&(\widehat{\op{Sym}}_{\tt A}(\tt M^*),{\tt A})\to(\widehat{\op{Sym}}_{\tt B}(\tt N^*),{\tt B})\\
(g^*,g_0):&(\widehat{\op{Sym}}_{\tt B}(\tt N^*),{\tt B})\to(\widehat{\op{Sym}}_{\tt C}(\tt L^*),{\tt C})
\end{split}
\end{equation*}
be maps in ${\tt CE}_{\op{cof}}(k)$. Maps $f_0$, $g_0$, and $f_0\circ g_0$ satisfy the 2-out-of-3 property as morphisms in ${\tt dgca}(k)$. Assuming  $f_0$, $g_0$ and $f_0\circ g_0$ are all weak equivalences, it is remains to show that  maps $f^l:N\to B\otimes_A M$, $g^l:L\to C\otimes_B N$, and $(f\circ g)^l:L\to  C\otimes_A M$ satisfies 2-out-for-three property. As $(f\circ g)^l=(C\otimes_B f^l)\circ g^l$; all that is required is to prove that $C\otimes_B f^l$ is a weak equivalence of $C$-modules if and only if $f^l$ is a weak equivalence of $B$-modules. Unit of the free-forgetful adjunction between $B$ and $C$-modules, $\eta_N:N\to  C\otimes_B N$ is a weak equivalence on cofibrant $B$-modules. With the the naturality of the unit,
\begin{center}
$$\begin{tikzcd}
N\arrow[r,"\eta_N"]\arrow[d,"f^l"]&\arrow[d,"C\otimes_B f^l"]C\otimes_B N\\
B\otimes_A M\arrow[r,"\eta_{B\otimes_A M}"]&C\otimes_B(B\otimes_A M)
\end{tikzcd}$$
\end{center}
the statement is a consequence the 2-out-of-3 property in the category of $B$-modules.
\subsection*{Cylinder object}
\begin{lemma}
Let $p: A\to B$ be a trivial fibrantion in ${\tt dgca}(k)$, and let $({\tt B}\langle m_i\rangle_{i\in I},d)$ be a cell complex in ${\tt Mod}(B)$. Then there exists a differential $\delta$ on ${\tt A}\langle m_i\rangle_{i\in I}$ such that: 
\begin{itemize}
\item $({\tt A}\langle m_i\rangle_{i\in I},\delta)$ is a cell complex in  ${\tt Mod}(A)$;
\item the projection $p:{\tt A}\langle m_i\rangle_{i\in I}\to{\tt B}\langle m_i\rangle_{i\in I}$ intertwines the differentials.
\end{itemize}
\end{lemma}
\begin{proof}
Lift of the differential is defined recursively in $I$. Base of the recursion is given by $\delta(m_1)=0$.
Assume that $\delta (m_j)$ is already defined for all $j<i$.  Under the identification ${\tt A} \langle m_i\rangle=\tt A$, its extension to $m_i$ is equivalently a map of $A$-modules $({\tt A} \langle m_i\rangle,d_A)\to ({\tt A}\langle m_j\rangle_{j<i},\delta)$. Such a map is given by the lifting diagram in ${\tt Mod}(A)$
\begin{center}
\begin{equation*}
\begin{tikzcd}
&({\tt A}\langle m_j\rangle_{j<i},\delta)\arrow[d,two heads,"\sim"]\\
({\tt A} \langle m_i\rangle,d_A)\arrow[r,"d\circ p"]\arrow[ru,dashed]&({\tt B}\langle m_j\rangle_{j<i},d).
\end{tikzcd}
\end{equation*}
\end{center}
\end{proof}

Let $(A,M)$ be a fibrant SH LR pair. Let
$$A\otimes_k A\to \op{Cyl}(A)\to A$$
be a cylinder decomposition in ${\tt dgca}(k)$. For ${\tt M=A}\langle m_i \rangle_{i\in I}$, suppressing differential from the notation, let
$${\tt A}\langle m_i \rangle_{i\in I}\to  {\tt A}\langle \overset{0}{m_i},\overset{1}{m_i},\overset{\tt I}{m_i} \rangle_{i\in I}\to {\tt A}\langle \overset{0}{m_i},\overset{1}{m_i} \rangle_{i\in I}$$ be the standard path decomposition in ${\tt Mod}(A)$.
 Let $\op{Cyl} ({\tt A})\langle {m_i}\rangle_{i\in I}$ be the lift  of ${\tt A}\langle {m_i}\rangle_{i\in I}$ provided by the above lemma.  $({{\tt A}\otimes_k {\tt A}})\langle \overset{0}{m_i},\overset{1}{m_i}\rangle_{i\in I}$ is a cell complex, with the differential dual to the linear part of the coproduct's differential. Denote $\op{Cyl} ({\tt A})\langle \overset{0}{m_i},\overset{1}{m_i}\rangle_{i\in I}=\op{Cyl}({\tt A})\otimes_{{\tt A}\otimes_k {\tt A}}{({\tt A}\otimes_k {\tt A})}\langle \overset{0}{m_i},\overset{1}{m_i}\rangle_{i\in I}$. Path object ${\tt A}\langle \overset{0}{m_i},\overset{1}{m_i},\overset{\tt I}{m_i} \rangle_{i\in I}$ is by the definition the $1$-shifted mapping cone of the map 
$\nabla: M\oplus M\to M$ defined by $\overset{1}m\mapsto{m}$, and  $\overset{0}m\mapsto{-m}$. Denote by $\phi$ the lift in
\begin{center}
\begin{equation*}
\begin{tikzcd}
&&\op{Cyl} ({\tt A})\langle {m_i}\rangle_{i\in I}\arrow[d,two heads,"\sim"]\\
\op{Cyl} ({\tt A})\langle \overset{0}{m_i},\overset{1}{m_i}\rangle_{i\in I}\arrow[r]\arrow[rru,dashed,"\phi"]\arrow[r,"\sim"]& {\tt A}\langle \overset{0}{m_i},\overset{1}{m_i}\rangle_{i\in I}\arrow[r]\arrow[r,"\nabla"] &{\tt A}\langle {m_i}\rangle_{i\in I}.
\end{tikzcd}
\end{equation*}
\end{center}
Denote  ${\op{Cyl}({\tt A})}\langle \overset{0}{m_i},\overset{1}{m_i},\overset{\tt I}{m_i} \rangle_{i\in I}=\op{Cone}(\phi)[1]$.  Dualizing, one arrives to a factorization of the diagonal map
$$(({{\tt A}\otimes_k {\tt A}})\langle\langle  \overset{0}{m_i^*} ,\overset{1}{m_i^*}\rangle\rangle_{i\in I},d_{\amalg}^0)\overset{i}{\rightarrow} ({\op{Cyl}({\tt A})}\langle \langle \overset{0}{m_i^*},\overset{1}{m_i^*},\overset{\tt I}{m_i^*} \rangle\rangle_{i\in I},d_{\op{Cyl}}^0)\overset{p}{\rightarrow}({\tt A}\langle \langle m_i^* \rangle\rangle_{i\in I}, d_{M^*}^0),$$
in which the map $p$ clearly dualizes to a quasi-isomorphism.
To finalize the construction, it suffices to determine the appropriate positive weight differentials $d_{\op{Cyl}}$ on the formal power series of the middle term. \bigskip

Denote the existing differentials by
$$(({{\tt A}\otimes_k {\tt A}})[[  \overset{0}{m_i^*} ,\overset{1}{m_i^*}]]_{i\in I}, d_{\amalg});\hspace{5pt} ({\tt A}[[ m_i^* ]]_{i\in I}, d_{M^*}).$$
Restriction on $\op{Cyl}(A)$ of a differential $d_{\op{Cyl}}^1$ which satisfies $d_{\op{Cyl}}^1\circ d_{\op{Cyl}(A)}^0 +d_{\op{Cyl}}^0\circ d_{\op{Cyl}}^1=0$ is equivalently a 1-cycle in the dg module of derivations on $\op{Cyl}(A)$ valued in $$({\op{ Cyl}({\tt A})}\langle \langle \overset{0}{m_i^*},\overset{1}{m_i^*},\overset{\tt I}{m_i^*} \rangle\rangle_{i\in I},d_{\op{Cyl}}^0)$$ -- the naive dual of ${\op{Cyl}({\tt A})}\langle \overset{0}{m_i},\overset{1}{m_i},\overset{\tt I}{m_i} \rangle_{i\in I}$, equivalently the projective limit of its dual differential pro-graded module. It is intertwined by $i$ with $d_{\amalg}^1|_A$ if it is contained in the (homotopy) fiber
\begin{center}
\begin{tikzcd}
\op{Fib}_{\op{Cyl}}^{1}\arrow[r,hook]\arrow[d, two heads]&\op{Der} (\op{ Cyl}({ A}),({{\tt Cyl}({\tt A})}\langle \langle \overset{0}{m_i^*},\overset{1}{m_i^*},\overset{\tt I}{m_i^*} \rangle\rangle_{i\in I},d_{\op{Cyl}}^0))\arrow[d,two heads]\\
k\langle d_{\amalg}^{1}|_A \rangle\arrow[r,hook,"i\circ -"]&\op{Der} (A\otimes_k A,({{\tt Cyl}({\tt A})}\langle \langle \overset{0}{m_i^*},\overset{1}{m_i^*},\overset{\tt I}{m_i^*} \rangle\rangle_{i\in I},d_{{\tt Cyl}}^0)),
\end{tikzcd}
\end{center}
and maps by the vertical arrow to $d_{\amalg}^1|_A$.
Similarly, restriction on $A$ of  $d_{M^*}^1$ (pre-composed with $p$) is an element of the homotopy fiber
\begin{center}
\begin{tikzcd}
\op{Fib}_{A}^{1}\arrow[r,hook]\arrow[d, two heads]&\op{Der} (\op{Cyl}({ A}),({\tt A}\langle \langle {m_i^*} \rangle\rangle_{i\in I},d_{M^*}^0))\arrow[d,two heads]\\
k\langle d_{\amalg}^{1}|_A \rangle\arrow[r,hook,"p\circ i\circ -"]&\op{Der} (A\otimes_k A,({\tt A}\langle \langle {m_i^*} \rangle\rangle_{i\in I},d_{M^*}^0)).
\end{tikzcd}
\end{center}
 Assuming that the natural map between cospans of the two homotopy fiber squares are quasi-isomorphisms, it follows that the induced map between the fibers is also a quasi-isomorphism. In the preimage in $\op{Fib}_{\op{Cyl}}^{1}$ of the cycle $d_{M^*}^1|_A$ in  $\op{Fib}_{A}^{A}$ is a cycle $d_{\op{Cyl}}^1$. Hence it satisfies $d_{\op{Cyl}}^1\circ d_{\op{Cyl}(A)} +d_{\op{Cyl}}^0\circ d_{\op{Cyl}}^1=0$ and it is intertwined by both  $i$ and $p$.

That the above-mentioned natural maps are indeed quasi-isomorphisms is obvious in the case when the index set $I$ is finite, as modules in which the derivations are valued are themselves graded-free finite modules, hence cofibrant. Map between them is the base change of a cofibrant module via a weak equivalence, hence weak equivalence, and derivations from semi-free dgca-s preserve weak equivalences. In general, one applies the isomorphism of the below lemma to the nods of above diagrams to conclude the same:
\begin{lemma}
Let $p:A\to B$ be a morphism in ${\tt dgca}(k)$, and let $M\in {\tt Mod}(B)$. Denoting $M^\vee=\op{Hom}_{{\tt Mod}(B)}(M,B),$ there is an isomorphism of dg vector spaces
$$\op{Der}(A,M^\vee)\cong \op{Hom}_{{\tt Mod}(B)}(M,\op{Der}(A,B)).$$
\end{lemma}
\begin{proof}
\begin{equation}
\begin{split}
\op{Der}(A,M^\vee)&\cong \op{Hom}_{{\tt Mod}(A)}(\mathbb{L}_A, \op{Hom}_{{\tt Mod}(B)}(M,B))\cong \op{Hom}_{{\tt Mod}(B)}(B\otimes_A\mathbb{L}_A, \op{Hom}_{{\tt Mod}(B)}(M,B))\\&
\op{Hom}_{{\tt Mod}(B)}(M\otimes_B(B\otimes_A\mathbb{L}_A), B)\cong \op{Hom}_{{\tt Mod}(B)}({M}, \op{Hom}_{{\tt Mod}(B)}(B\otimes_A\mathbb{L}_A,B))\\
&\cong \op{Hom}_{{\tt Mod}(B)}({M}, \op{Hom}_{{\tt Mod}(A)}(\mathbb{L}_A,B)) \cong \op{Hom}_{{\tt Mod}(B)}({M}, \op{Der}(A,B)).
\end{split}
\end{equation}
\end{proof}

Restriction on  $({\op{ Cyl}({\tt A})}\langle \langle \overset{0}{m_i^*},\overset{1}{m_i^*},\overset{\tt I}{m_i^*} \rangle\rangle_{i\in I}$ of a differential $d_{\op{Cyl}}^1$ is defined as follows. First, one defines a provisional differential by
$$\delta_{\op{Cyl}}^1(\overset{0}{m_i})=d_{\amalg}^1(\overset{0}{m_i});\hspace{5pt} \delta_{\op{Cyl}}^1(\overset{1}{m_i})=d_{\amalg}^1(\overset{1}{m_i});\hspace{5pt}\delta_{\op{Cyl}}^1(\overset{\tt I}{m_i})=0;\hspace{5pt}\delta_{\op{Cyl}}^1(a\cdot \overset{0,1,\tt I}{m_i})=d_{\op{Cyl}}^1|_A a\cdot \overset{0,1,\tt I}{m_i}+a\cdot \delta_{\op{Cyl}}^1(\overset{0,1, \tt I}{m_i})$$
Clearly, the provisional differential intertwines $i$ and $p$. A direct verification shows that the graded commutator $[\delta_{\op{Cyl}}^1,d_{\op{Cyl}}^0]$ is $\op{Cyl}(A)$-linear:
\begin{equation*}
\begin{split}
[\delta_{\op{Cyl}}^1,d_{\op{Cyl}}^0]am&=[d_{\op{Cyl}}^1\circ d_{\op{Cyl}}^0](a)\cdot m +a\cdot [\delta_{\op{Cyl}}^1,d_{\op{Cyl}}^0](m)+(-1)^{|a|-1}(d_{\op{Cyl}}^1(a) d_{\op{Cyl}}^0(m)+d_{\op{Cyl}}^0(a) \delta_{\op{Cyl}}^1(m))\\
&+(-1)^{|a|}(d_{\op{Cyl}}^0(a) \delta_{\op{Cyl}}^1(m)+d_{\op{Cyl}}^1(a) d_{\op{Cyl}}^0(m))=a\cdot [\delta_{\op{Cyl}}^1,d_{\op{Cyl}}^0](m)
\end{split}
\end{equation*}
Given a (differential) graded-flat module $N$, denote by ${N^*}^{\widehat{\odot}2}$ the (differential) pro-graded dual of the symmetric tensor product $N\odot N$. If one can find a $\tt{Cyl}( A)$-linear map 
$$\Delta:{\op{ Cyl}({\tt A})}\langle \langle \overset{0}{m_i^*},\overset{1}{m_i^*},\overset{\tt I}{m_i^*} \rangle\rangle_{i\in I}\to({\op{ Cyl}({\tt A})}\langle \langle \overset{0}{m_i^*},\overset{1}{m_i^*},\overset{\tt I}{m_i^*} \rangle\rangle_{i\in I})^{\widehat{\odot}2}[-1]$$
which vanishes when composed by either of the maps $i$ and $p$, and such that $[\Delta+\delta_{\op{Cyl}}^1,d_{\op{Cyl}}^0]=0$, then $\Delta+\delta_{\op{Cyl}}^1$ is a required restriction of $d_{\op{Cyl}}^1$.

Consider the following diagram of graded vector spaces of $\op{Pro}({\tt Mod}(\op{Cyl}A))$-morphisms:
\begin{center}
\begin{tikzcd}
\op{Ker}_{u}\arrow[r,hook]\arrow[d]&\op{Hom}^{\bullet} (\op{ Cyl}{ A}\langle \langle \overset{0,1,\tt I}{m_i^*}\rangle\rangle,(\op{ Cyl}{ A}\langle \langle \overset{0,1,\tt I}{m_i^*}\rangle\rangle)^{\widehat{\odot}2})\arrow[r,two heads]\arrow[d]&\op{Hom}^{\bullet} (\op{ Cyl}{ A}\langle \langle \overset{0,1}{m_i^*}\rangle\rangle,(\op{ Cyl}{ A}\langle \langle \overset{0,1,\tt I}{m_i^*}\rangle\rangle)^{\widehat{\odot}2})\arrow[d]\\
\op{Ker}_{d}\arrow[r,hook]&\op{Hom}^{\bullet} (\op{ Cyl}{ A}\langle \langle \overset{0,1,\tt I}{m_i^*}\rangle\rangle,({ A}\langle \langle {m_i^*}\rangle\rangle)^{\widehat{\odot}2})\arrow[r,two heads]&\op{Hom}^{\bullet} (\op{ Cyl}{ A}\langle \langle \overset{0,1}{m_i^*}\rangle\rangle,({ A}\langle \langle {m_i^*}\rangle\rangle)^{\widehat{\odot}2})
\end{tikzcd}
\end{center}
Any degree 1 element of $\op{Ker}_{u}$ whose image in $\op{Ker}_{d}$ equals zero is a map $\Delta$ which vanishes when composed by either of the maps $i$ and $p$. If $I$ is finite, one is dealing with perfect $\op{Cyl }A$ and $A$-modules which safely dualize without passing to the pro-category, and the above diagram represents a weak equivalence of homotopy fiber sequences in ${\tt Mod}(\op{Cyl }A)$ . 
$\op{Cone}(\op{Ker}_{u}\to \op{Ker}_{d})$ is acyclic, and $-[\delta^1_{\op{Cyl}},d^0_{\op{Cyl}}]$ is a cycle, hence a boundary. Thus there exist $\Delta'\in \op{Ker}_{u}$, $D\in \op{Ker}_{d}$ such that $\Delta'$ maps to $[D,d^0_{M^*}]$, and
$[\Delta',d^0_{\op{Cyl}}]=-[\delta^1_{\op{Cyl}},d^0_{\op{Cyl}}]$.
 $D$ lifts to a map $\overline{D}\in \op{Ker}_{u}$. Namely, as it vanishes on generators labeled by $0$ and $1$, it suffices to lift its restriction to the graded submodule generated by $\tt {I}$-labeled terms and extend the lift by zero. Finally, $\Delta=\Delta'-[\overline{D},d^0_{\op{Cyl}}]$ vanishes when composed by either of the maps $i$ and $p$, and $[\Delta,d^0_{\op{Cyl}}]=-[\delta^1_{\op{Cyl}},d^0_{\op{Cyl}}]$.

The general case requires dualization
\begin{center}
\begin{tikzcd}[scale cd=0.9]
\overline{\op{Ker}}_{u}\arrow[r,hook]\arrow[d,"\sim"]&\op{Hom}_{{\tt Mod}(\op{Cyl }A)}^{\bullet} ((\op{ Cyl}{ A}\langle  \overset{0,1,\tt I}{m_i}\rangle)^{{\odot}2},\op{ Cyl}{ A}\langle  \overset{0,1,\tt I}{m_i}\rangle)\arrow[r,two heads]\arrow[d,"\sim"]
&\op{Hom}_{{\tt Mod}(\op{Cyl }A)}^{\bullet} ((\op{ Cyl}{ A}\langle \overset{0,1,\tt I}{m_i}\rangle)^{{\odot}2},\op{ Cyl}{ A}\langle \overset{0,1}{m_i}\rangle)\arrow[d,"\sim"]\\
\overline{\op{Ker}}_{m}\arrow[r,hook]\arrow[d,"\sim"]&\op{Hom}_{{\tt Mod}(A)}^{\bullet} ((A \langle \overset{0,1,\tt I}{m_i}\rangle)^{{\odot}2},A\langle  \overset{0,1,\tt I}{m_i}\rangle)\arrow[r,two heads]\arrow[d,"\sim"]&\op{Hom}_{{\tt Mod}(A)}^{\bullet} ((A\langle \overset{0,1,\tt I}{m_i}\rangle)^{{\odot}2},A\langle \overset{0,1}{m_i}\rangle)\arrow[d,"\sim"]\\
\overline{\op{Ker}}_{d}\arrow[r,hook]&\op{Hom}_{{\tt Mod}(A)}^{\bullet} (({ A}\langle {m_i}\rangle)^{{\odot}2},{ A} \langle \overset{0,1,\tt I}{m_i}\rangle)\arrow[r,two heads]&\op{Hom}_{{\tt Mod}(A)}^{\bullet} (({ A}\langle {m_i}\rangle)^{{\odot}2},{ A}\langle \overset{0,1}{m_i}\rangle)
\end{tikzcd}
\end{center}
to obtain again weak equivalences between homotopy fiber sequences.

As before, $\op{Cone}(\overline{\op{Ker}}_{u}\to\overline{\op{Ker}}_{d})$ is 
acyclic, and $[(\delta^1_{\op{Cyl}})^*,(d^0_{\op{Cyl}})^*]$
is a cycle, hence a boundary. Thus there exist ${\Delta'}^*\in \overline{\op{Ker}}_{u}$, $D^*\in \overline{\op{Ker}}_{d}$ 
such that ${\Delta'}^*$ maps to $[D^*,{d^0}^*]$, and
$[{\Delta'}^*,{d^0_{\op{Cyl}}}^*]=-[{\delta^1_{\op{Cyl}}}^*,{d^0_{\op{Cyl}}}^*]$.
 $D^*$ lifts to a map ${\overline{D}}^*\in \overline{\op{Ker}}_{u}$. Finally, the dual $\Delta$  of $\Delta^*={\Delta'}^*-[{\overline{D}}^*,{d^0_{\op{Cyl}}}^*]$, vanishes when composed by either of the maps $i$ and $p$, and $$[\Delta,d^0_{\op{Cyl}}]^*=[\Delta^*,{d^0_{\op{Cyl}}}^*]=-[{\delta^1_{\op{Cyl}}}^*,{d^0_{\op{Cyl}}}^*]=[{\delta^1_{\op{Cyl}}},{d^0_{\op{Cyl}}}]^*,$$ hence $[\Delta,d^0_{\op{Cyl}}]=[{\delta^1_{\op{Cyl}}},{d^0_{\op{Cyl}}}].$ This finalizes the induction's basis. \bigskip

For a fixed $n\in \mathbb{N}$, assume differentials  ${d^0_{\op{Cyl}}},\ldots {d^n_{\op{Cyl}}}$ such that for all $k=1,\ldots, n$
\begin{itemize}
\item  $p$ and $i$ intertwine weight $k$ differentials;
\item $\sum_{l=0}^k d^l_{\op{Cyl}}\circ d^{k-l}_{\op{Cyl}}=0$
\end{itemize}
have been found.

Notice first that it follows from Jacobi identity in the graded Lie algebra of derivations (with cohomological grading) that $[\sum_{l=1}^n d^l_{\op{Cyl}}\circ d^{n+1-l}_{\op{Cyl}},d^0_{{\op{Cyl}}}]=0$. Indeed,
\begin{equation*}
\begin{split}
[\sum_{l=1}^n d^l_{\op{Cyl}}\circ d^{n+1-l}_{\op{Cyl}},d^0_{{\op{Cyl}}}]=&\frac{1}{2}\sum_{l=1}^n[ [d^l_{\op{Cyl}}, d^{n+1-l}_{\op{Cyl}}],d^0_{{\op{Cyl}}}=\sum_{\substack{k+l=n+1\\k,l\geq 1}}[ d^{k}_{\op{Cyl}}, [d^{l}_{\op{Cyl}},d^0_{{\op{Cyl}}}]]\\
=&
-\frac{1}{2}\sum_{\substack{k+l+m=n+1\\k,l,m\geq 1}}[ d^{k}_{\op{Cyl}}, [d^{l}_{\op{Cyl}},d^m_{{\op{Cyl}}}]]=0.
\end{split}
\end{equation*}
Hence, $-\sum_{l=1}^n d^l_{\op{Cyl}}\circ d^{n+1-l}_{\op{Cyl}}|_A$ is a cycle in $\op{Fib}_{\op{Cyl}}^{n+1}$ which is mapped by the vertical arrow in the below diagram to $[d_{\amalg}^{n+1}|_A,d_{\amalg}^{0}]$.
\begin{center}
\begin{tikzcd}
\op{Fib}_{\op{Cyl}}^{n+1}\arrow[r,hook]\arrow[d, two heads]&\op{Der} (\op{ Cyl}({ A}),({{\tt Cyl}({\tt A})}\langle \langle \overset{0,1,\tt I}{m_i^*} \rangle\rangle_{i\in I}^{\widehat{\odot} (n+1)},d_{\op{Cyl}}^0))\arrow[d,two heads]\\
(k\langle d_{\amalg}^{n+1}|_A ,[d_{\amalg}^{n+1}|_A,d_{\amalg}^{0}]\rangle,[-,d_{\amalg}^{0}])   \arrow[r,hook,"i\circ -"]&\op{Der} (A\otimes_k A,({{\tt Cyl}({\tt A})}\langle \langle \overset{0,1,\tt I}{m_i^*} \rangle\rangle_{i\in I}^{\widehat{\odot} (n+1)},d_{{\op{Cyl}}}^0)),
\end{tikzcd}
\end{center}
Its image in $\op{Fib}_{M^*}^{n+1}$ (see the below diagram) is a boundary
$$-\sum_{l=1}^n d^l_{M^*}\circ d^{n+1-l}_{M^*}|_A= [d^{n+1}_{M^*}|_A,d^0_{M^*}].$$
\begin{center}
\begin{tikzcd}
\op{Fib}_{M^*}^{n+1}\arrow[r,hook]\arrow[d, two heads]&\op{Der} (\op{Cyl}({ A}),({\tt A}\langle \langle {m_i^*} \rangle\rangle_{i\in I}^{\widehat{\odot} (n+1)},d_{M^*}^0))\arrow[d,two heads]\\
(k\langle d_{\amalg}^{n+1}|_A ,[d_{\amalg}^{n+1}|_A,d_{\amalg}^{0}]\rangle,[-,d_{\amalg}^{0}])  \arrow[r,hook,"p\circ i\circ -"]&\op{Der} (A\otimes_k A,({\tt A}\langle \langle {m_i^*} \rangle\rangle_{i\in I}^{\widehat{\odot} (n+1)},d_{M^*}^0)).
\end{tikzcd}
\end{center}

Consequently, $d^{n+1}_{M^*}-\sum_{l=1}^n d^l_{\op{Cyl}}\circ d^{n+1-l}_{\op{Cyl}}|_A$ is a cycle in $\op{Cone}(\op{Fib}_{\op{Cyl}}^{n+1}\to \op{Fib}_{M^*}^{n+1})$, hence a boundary. Consequently, there exist derivations $D\in \op{Fib}_{M^*}^{n+1}$, $\Delta\in\op{Fib}_{\op{Cyl}}^{n+1}$ such that $\Delta$ maps to $[D,d^0_{M^*}]+d^{n+1}_{M^*}$, and $[\Delta,d^0_{\op{Cyl}}]=-\sum_{l=1}^n d^l_{\op{Cyl}}\circ d^{n+1-l}_{\op{Cyl}}|_A$.

Denote by $\overline{D}$ a lift of $D$ to $\op{Fib}_{\op{Cyl}}^{n+1}$.  $d^{n+1}_{\op{Cyl}}|_A:=\Delta+[\overline{D},d^0_{\op{Cyl}}]$ is intertwined by $i$ and $p$ with the corresponding weight $n+1$ diffrentials and such that $\sum_{l=0}^{n+1} d^l_{\op{Cyl}}\circ d^{k-l}_{\op{Cyl}}|_A=0.$

Restriction on  $({\op{ Cyl}({\tt A})}\langle \langle \overset{0}{m_i^*},\overset{1}{m_i^*},\overset{\tt I}{m_i^*} \rangle\rangle_{i\in I}$ of a differential $d_{\op{Cyl}}^{n+1}$ is defined as follows. First, one defines a provisional differential by
$$
\begin{matrix*}[l]  
 \delta_{\op{Cyl}}^{n+1}(\overset{0}{m_i})=d_{\amalg}^{n+1}(\overset{0}{m_i}); &  \delta_{\op{Cyl}}^{n+1}(\overset{1}{m_i})=d_{\amalg}^{n+1}(\overset{1}{m_i}); \\        
 \delta_{\op{Cyl}}^{n+1}(\overset{\tt I}{m_i})=0; & \delta_{\op{Cyl}}^{n+1}(a\cdot \overset{0,1,\tt I}{m_i})=d_{\op{Cyl}}^{n+1}|_A a\cdot \overset{0,1,\tt I}{m_i}+a\cdot \delta_{\op{Cyl}}^{n+1}(\overset{0,1, \tt I}{m_i})  .       
\end{matrix*} 
$$
Clearly, the provisional differential intertwines $i$ and $p$, and $[\delta_{\op{Cyl}}^{n+1},d_{\op{Cyl}}^0]$ is $A$-linear.
If one can find an $\tt{Cyl}( A)$-linear map 
$$\Delta:{\op{ Cyl}({\tt A})}\langle \langle \overset{0}{m_i^*},\overset{1}{m_i^*},\overset{\tt I}{m_i^*} \rangle\rangle_{i\in I}\to({\op{ Cyl}({\tt A})}\langle \langle \overset{0}{m_i^*},\overset{1}{m_i^*},\overset{\tt I}{m_i^*} \rangle\rangle_{i\in I})^{\widehat{\odot}(n+1)}[-1]$$
which vanishes when composed by either of the maps $i$ and $p$, and such that 
\begin{equation}\label{step}[\Delta+\delta_{\op{Cyl}}^{n+1},d_{\op{Cyl}}^0]=-\sum_{l=1}^n d^l_{\op{Cyl}}\circ d^{n+1-l}_{\op{Cyl}}=-\frac{1}{2}\sum_{l=1}^n [d^l_{\op{Cyl}}, d^{n+1-l}_{\op{Cyl}}]\end{equation} then $\Delta+\delta_{\op{Cyl}}^1$ is a required restriction of $d_{\op{Cyl}}^1$. Notice that $[d^l_{\op{Cyl}}, d^{n+1-l}_{\op{Cyl}}]$ is $\tt{Cyl}( A)$-linear.

Consider the following diagram of graded vector spaces of $\op{Pro}({\tt Mod}(\op{Cyl}A))$-morphisms:
\begin{center}
\begin{tikzcd}[scale cd=0.95]
\op{Ker}_{u}\arrow[r,hook]\arrow[d]&\op{Hom}^{\bullet} (\op{ Cyl}{ A}\langle \langle \overset{0,1,\tt I}{m_i^*}\rangle\rangle,(\op{ Cyl}{ A}\langle \langle \overset{0,1,\tt I}{m_i^*}\rangle\rangle)^{\widehat{\odot}{n+1}})\arrow[r,two heads]\arrow[d]&\op{Hom}^{\bullet} (\op{ Cyl}{ A}\langle \langle \overset{0,1}{m_i^*}\rangle\rangle,(\op{ Cyl}{ A}\langle \langle \overset{0,1,\tt I}{m_i^*}\rangle\rangle)^{\widehat{\odot}{n+1}})\arrow[d]\\
\op{Ker}_{d}\arrow[r,hook]&\op{Hom}^{\bullet} (\op{ Cyl}{ A}\langle \langle \overset{0,1,\tt I}{m_i^*}\rangle\rangle,({ A}\langle \langle {m_i^*}\rangle\rangle)^{\widehat{\odot}{n+1}})\arrow[r,two heads]&\op{Hom}^{\bullet} (\op{ Cyl}{ A}\langle \langle \overset{0,1}{m_i^*}\rangle\rangle,({ A}\langle \langle {m_i^*}\rangle\rangle)^{\widehat{\odot}{n+1}})
\end{tikzcd}
\end{center}
Any degree 1 element of $\op{Ker}_{u}$ whose image in $\op{Ker}_{d}$ equals zero is a map $\Delta$ which vanishes when composed by either of the maps $i$ and $p$. If $I$ is finite, one is dealing with perfect $\op{Cyl }A$ and $A$-modules which safely dualize without passing to the pro-category, and the above diagram represents a weak equivalence of homotopy fiber sequences in ${\tt Mod}(\op{Cyl }A)$ . 
$\op{Cone}(\op{Ker}_{u}\to \op{Ker}_{d})$ is acyclic, and $$-[\delta^{n+1}_{\op{Cyl}},d^0_{\op{Cyl}}]-\frac{1}{2}\sum_{l=1}^n [d^l_{\op{Cyl}}, d^{n+1-l}_{\op{Cyl}}]$$ is a cycle, hence a boundary. Thus there exist $\Delta'\in \op{Ker}_{u}$, $D\in \op{Ker}_{d}$ such that $\Delta'$ maps to $[D,d^0]$, and
$$[\Delta',d^0_{\op{Cyl}}]=-[\delta^{n+1}_{\op{Cyl}},d^0_{\op{Cyl}}]-\frac{1}{2}\sum_{l=1}^n [d^l_{\op{Cyl}}, d^{n+1-l}_{\op{Cyl}}].$$
 $D$ lifts to a map $\overline{D}\in \op{Ker}_{u}$. Finally $\Delta=\Delta'-[\overline{D},d^0_{\op{Cyl}}]$, vanishes when composed by either of the maps $i$ and $p$, and satisfies the equation (\ref{step}). The general case is straight-forward.

\vfill
\Addresses

\end{document}